\documentclass[a4paper]{amsart}   

\usepackage{amssymb}
\usepackage{MnSymbol}
\usepackage{bbm}
\usepackage{color}
\usepackage{graphicx}

\newtheorem{theorem}{Theorem}
\newtheorem{lemma}[theorem]{Lemma}
\newtheorem{proposition}[theorem]{Proposition}
\newtheorem{corollary}[theorem]{Corollary}
\theoremstyle{definition}
\newtheorem{example}{Example} 
\newtheorem{definition}{Definition}    

\newcommand{\RR}{\mathbb{R}}
\newcommand{\ZZ}{\mathbb{Z}}
\newcommand{\FF}{\mathcal{F}}
\newcommand{\LL}{\mathcal{L}}
\newcommand{\XX}{\mathcal{X}}
\newcommand{\holoL}{H_\LL}
\newcommand{\gamaL}{\Gamma_{\LL}}
\newcommand{\Jup}{\widetilde{J}^\uparrow}
\newcommand{\one}{\mathbbm{1}}

\title[Projections of Patterns and Mode Interactions \today]{Projections of Patterns and Mode Interactions \\\today}

\author[ S.B.S.D.Castro, I.S.Labouriau, J.F. Oliveira \today]{ Sofia B.S.D. Castro, Isabel S. Labouriau, Juliane F. Oliveira}
\address{Centro de Matem\'atica da Universidade do Porto, Rua do Campo Alegre 687, 4169-007, Porto, Portugal \\ and
J.F. Oliveira, I.S.Labouriau --- Departamento de Matem\'atica, Faculdade de Ci\^encias da Universidade do Porto,Portugal\\
S.B.S.D.Castro --- Faculdade de Economia, Universidade do Porto, Rua Dr. Roberto
Frias, 4200-464, Porto, Portugal}


\begin{document}   

\begin{abstract}
We study solutions of bifurcation problems with periodic boundary conditions, with periods in an $n+1$-dimensional lattice and their projection  into $n$-dimensional space through integration of the last variable.
We show that generically the projection of  a single mode solution is a mode interaction.
This can be applied to the study of black-eye patterns.
\end{abstract}

\maketitle

\section{Introduction}
%To be written. 
%
%Planforms and patterns.
%
%Black-eye patterns. Are they projections? Are they mode interactions?
The aim of this article is to contribute to the understanding of models for the formation of periodic patterns in a thin layer.
Experiments in thin layers are often modelled as 2-dimensional rather than 3-dimensional problems. The observation of results in such experiments is achieved through a 2-dimensional medium which is the top layer. In such instances, the observed patterns are 2-dimensional projections of the 3-dimensional objects in the thin layer.

An early example consists of the results found in reaction-diffusion experiments in the Turing instability regime, \cite{turing52}. Typically, a reaction occurs in a thin layer of gel, fed by diffusion from one or two faces with chemicals contained in stirred tanks. The first reaction that provided Turing instabilities was in experiments on the chlorite-iodide-malonic-acid (CIMA) reaction, \cite{castets90}.
The pattern itself and its observed state can occur in different dimensions, see \cite{kedubo91, ouswi95, gomes99}. This happens for instance when an experiment is done in a 3-dimensional medium but the patterns are only observed on its surface, a 2-dimensional object, see \cite[Section 4]{kedubo91}. The interpretation of this 2-dimensional outcome is subject to discussion: the black-eye pattern observed by \cite{ouswi95} has been explained both as a mode interaction in 2 dimensions in \cite{gunaouswi94,Zhou2002}, and described as a projection of a fully 3-dimensional pattern in \cite{gomes99}. 
The main challenge is to choose one of these descriptions. 
Our results in Section \ref{secIrreducible} indicate that these descriptions may coincide. We show that often the projection of a 3-dimensional single mode (the explanation provided in \cite{gomes99}) bifurcation is a 2-dimensional mode interaction (the explanation supported by \cite{gunaouswi94,Zhou2002}). The comments of \cite{Zhou2002} concerning the monolayers that do or do not support the occurrence and observation of a black-eye pattern can be related to the width of the projection band (defined in Subsection \ref{subsecProjections}).
%...... to find out if the experimental black-eye pattern and its representation as projection are the same.

Projection has proved to be a good way to model experiments. What is believed to be the first evidence of projection on the CIMA reaction can be found in \cite[Chapter 13]{winfree01}. The author gives details about the geometry of the formation of wave patterns in malonic-acid reaction performed on sufficiently thin layers. In \cite{kedubo91} the authors conducted experiments on the CIMA reaction and aimed at describing experimental observations of spontaneous symmetry breaking phenomena associated with steady-state instabilities. In \cite[Section 4]{kedubo91}, the authors highlight the natural environment we must consider when we carry out CIMA reactions, in particular they state that all of their observations were based on projection of 3-dimensional structures. Moreover, the regions where Turing patterns are observed are associated by projection to a body-centred cubic lattice. More discussion on this can be found in \cite{dulos96, dekepper00, borckmans02, szalai15}. As reaction and diffusion progress bifurcations may occur. Bifurcating solutions are detected by observation of the top, 2-dimensional, layer. This poses the question of how solutions of a 3-dimensional object appear in their 2-dimensional observation.
	 
Although we are  motivated by the 2-dimensional projection of 3-dimensional objects, we study the relation between patterns in a $(n+1)$-dimensional space and their  projection in a $n$-dimensional space. 
Since a pattern consists of the level curves of periodic functions, we characterise the space of projected periodic functions, describe the projection of orbits in the dual lattice and provide a decomposition of the span of projected irreducible subspaces.
We achieve this by an algebraic approach to the relevant symmetry-related objects that contribute to the description of patterns.
After some preliminary explanations, the main results are stated at the end of Section~\ref{secPreliminary}.

Information about which symmetries are preserved under projection has been obtained by Labouriau and Pinho in \cite{labpin06,pinlab14}. In Oliveira {\em et al.} \cite{Olicalab15}, we describe which symmetries can lead to a projected function with hexagonal symmetry. In studying the dynamics of $(n+1)$-dimensional patterns by observing their $n$-dimensional projection, it is desirable to extract information from the symmetries that can be seen in the $n$-dimensional space directly. The present article contributes to this point.

\section{Bifurcation Problems with Euclidean Symmetry}\label{secPreliminary}
This section contains a rigorous formulation of the setting  of the article and a statement of its main results.
For this, definitions and basic results that are used in the article are stated here.
More information on crystallographic groups may be found in Miller \cite{miller72}, Armstrong \cite[chapters 24 to 26]{armstrong88}, Senechal \cite{senechal96}, and the International Tables for Crystallography (ITC) volume A  \cite{int-tab-a}. 
Results on  equivariant bifurcation theory can be found in  Golubitsky and Stewart's book \cite{goluste02}. For mode interactions we refer the reader to Castro \cite{castro93, castro94, castro95} and \cite[Ch XIX and XX]{goluste12}.

Let  $E(n+1)$ be the \emph{Euclidean group} of all isometries on $\RR^{n+1}$, that may be described as the semi-direct sum $E(n+1)\cong \RR^{n+1}\dotplus O(n+1)$. 
Its elements are ordered pairs, $(v,\delta)$, where $v \in \RR^{n+1}$ is a translation and $\delta$ is an element of the orthogonal group $O(n+1)$. 
The group operation is $ (v_1,\delta_1)\cdot (v_2,\delta_2)=(v_1+\delta_1 v_2,\delta_1\delta_2) $.
Elements $(v,\delta)\in E(n+1)$ act on functions $f:\RR^{n+1}\rightarrow \RR$ by 
$\left((v,\delta)\cdot f \right) (x)=f(\delta^{-1}(x-v))$.

Consider a one parameter family of partial differential equations
\begin{equation}\label{int}
\frac{\partial u}{\partial t}(x,t) = \FF(u(x,t),\lambda)
\end{equation}
where $\FF:\XX\times \RR\longrightarrow \mathbf{\mathcal{Y}}$ is an operator between suitable function spaces $\XX$ and $\mathcal{Y}$ and $\lambda \in \RR$ is a bifurcation parameter.
The function $u: \RR^{n+1}\times\RR\rightarrow \RR$ in $\XX$ is a function of a spatial variable $x \in \RR^{n+1}$ and of time $t$.

Suppose that $\FF$ is \emph{equivariant} under the Euclidean group $E(n+1)$, that is,
$$
\gamma\cdot\FF(u,\lambda) = \FF(\gamma\cdot u, \lambda), \ \ \text{for all} \ \ \gamma \in E(n+1).
$$
Equilibria of \eqref{int} are time independent solutions $u(x)$ that satisfy $\FF(u(x),\lambda) = 0$.
They are also called steady-states.

We give a brief description of  a standard method to use  symmetries to study the way steady-states  in \eqref{int} bifurcate  from the  trivial solution $\FF(0,\lambda)\equiv 0$.
Details may be found in  \cite{digolu92,goluste02}.
%The first step is to  restrict the problem to a subspace  $\XX_{\LL}$ of functions periodic under an \emph{$(n+1)$-dimensional lattice}, $\LL\subset  \RR^{n+1}$, a set generated over the integers by $n+1$ linearly independent elements $l_{1},\ldots,l_{n+1} \in \RR^{n+1}$. 
The first step is to consider an \emph{$(n+1)$-dimensional lattice}, $\LL\subset  \RR^{n+1}$, a set generated over the integers by $n+1$ linearly independent elements $l_{1},\ldots,l_{n+1} \in \RR^{n+1}$, and to restrict the problem to a space  $\XX_{\LL}$ of functions periodic under $\LL$. 

%We write: \[\LL = \langle l_{1},\ldots,l_{n+1}\rangle_{\ZZ}.\] 

There are natural symmetries in this space that we aim to explore to simplify the problem by applying the theory of bifurcation with symmetry.
The restricted problem is equivariant under the action of the group $\gamaL$, the largest group constructed from $E(n+1)$ that leaves the space $\XX_{\LL}$ invariant. 
Translations map $\XX_{\LL}$ into itself, the action  of translations on $\XX_{\LL}$ is that of the torus 
$\mathbb{T}^{n+1}=\RR^{n+1}/\LL$.
The subgroup of $O(n+1)$ that maps $\LL$ to itself, and thus leaves $ \XX_\LL$  invariant, forms the \emph{holohedry} of $ \LL$ 
(called \emph{point symmetry of the lattice} in \cite{int-tab-a}) and is denoted by $\holoL$. 
The holohedry is always  finite, see \cite[Sec. 2.4.2]{senechal96}. 
The group $\gamaL$ is thus compact and  can  be written as a semi-direct sum
$$
\gamaL = \RR^{n+1}/\LL  \dotplus  \holoL=
\mathbb{T}^{n+1}  \dotplus  \holoL .
$$

The next step is to analise the bifurcation of solutions to the restricted problem.
A steady-state bifurcation at $\lambda=0$ occurs when the linearisation of $\FF(0,0)$ has a non-trivial kernel 
$V\subset \XX_\LL$.
The kernel is always $\gamaL$-invariant  in the sense that if $f\in V$ then $\gamma\cdot f\in V$ for all $\gamma\in\gamaL$.
Since $\gamaL$ is compact, we expect $V$ to be finite dimensional.
This simplifies the problem considerably, as the  study of bifurcating solutions is reduced to a   
$\gamaL$-equivariant bifurcation problem defined in a finite dimensional space $V$.
It also allows the use of equivariant bifurcation theory, where solution branches  can then  be obtained using the Equivariant Branching Lemma  \cite[Lemma 1.31]{goluste02}, independently proved by Cicogna \cite{Cicogna} and Vanderbauwhede \cite{Vanderbauwhede}.

Next we provide some elementary definitions.
Let $\Gamma$ be a symmetry group of isometries acting on a vector space $V$.
The \emph{orbit} of $v\in V$ under this action is the set $\Gamma\cdot v=\{\gamma v:\ \gamma\in\Gamma\}$.
The \emph{isotropy subgroup} $\Sigma_v$ of $v\in V$ is $\Sigma_v=\{\gamma\in\Gamma:\ \gamma v=v\}$.
%The maximum reduction can be obtained for a symmetry group $\Gamma$ acting on a finite-dimensional vector space $V$ as follows.
We say that a vector subspace $W\subset V$  is \emph{$\Gamma$-invariant} if  $\gamma\in\Gamma$ and  $v\in W$ implies $\gamma v\in W$.
A $\Gamma$-invariant subspace  $W\subset V$ is \emph{$\Gamma$-irreducible} if $W$ contains no proper non-trivial $\Gamma$-invariant subspace. In this case, a problem defined in $W$ cannot be reduced further.

Consider a $\Gamma$-equivariant bifurcation problem :
\begin{equation}\label{eq1}
\frac{dz}{dt} = g(z,\lambda)
\qquad\mbox{with} \qquad g(0,\lambda)=0 
%, \ \ \  .
\end{equation}
with $g:V\times\RR\longrightarrow V$ satisfying $g(\gamma z,\lambda)=\gamma g(z,\lambda)$ for all $\gamma\in\Gamma$ and all $(z,\lambda)\in V\times\RR$.	
\begin{definition}\label{mode} 
The bifurcation problem \eqref{eq1} has:
\begin{itemize}
\item
 a \emph{single mode bifurcation} at $\lambda=0$
if the  kernel of $(dg)_{0,0}$ is non-trivial and  $\Gamma$-irreducible;
\item 
an \emph{$r$-mode interaction} at $\lambda=0$ if the kernel of $(dg)_{0,0}$ can be decomposed as the direct sum of $r>1$ non-trivial components that are $\Gamma$-irreducible.
\end{itemize}
\end{definition}

A characterisation of $\gamaL$-irreducible subspaces of $\XX_\LL$ is given by 
Dionne and Golubitsky \cite{digolu92}, as follows.

Let $\langle\cdot,\cdot\rangle$ be the usual inner product in $\RR^{n+1}$. We assume that all the functions 
%$f:\RR^{n+1} \rightarrow \RR$ 
in $\XX_\LL$ admit a Fourier expansion in terms of the \emph{waves} 
$$
\omega_{k}(x,y) = exp(2\pi i \langle k,(x,y)\rangle),
$$
where $k$ is a \emph{wave vector} in the dual lattice, 
$\LL^{*} = \lbrace k \in \RR^{n+1}; \ \langle k, l_{i} \rangle \in \ZZ, \ i = 1, \cdots, \ n+1 \rbrace$, of $\LL$
and with the notation $(x,y)\in\RR^{n+1}$, $x\in\RR^n$, $y\in\RR$.

Subspaces of $\XX_\LL$ that are $\gamaL$-irreducible must be, in particular, $\mathbb{T}^{n+1}$-invariant.
Given $\ell \in \mathbb{T}^{n+1}$ and $k \in \LL^{*}$ we have 
$\ell\cdot \omega_{k}(x,y) = \omega_{k}(-\ell)\omega_{k}(x,y)$.
Hence, the two-dimensional subspaces 
\begin{equation}\label{V_{k}}
V_{k} =  \lbrace Re(z\omega_{k}(x,y)); \ z \in \mathbb{C}\rbrace
\end{equation}
are $\mathbb{T}^{n+1}$-irreducible. Morever, $V_{k}$ and $V_{k^\prime}$ are distinct $\mathbb{T}^{n+1}$ representations if $k \neq \pm k^\prime$. 

The action of $\gamma\in\holoL$ on $\omega_k(x,y)$ satisfies
\begin{equation}\label{gamaOfOmega}
\gamma\omega_k(x,y)= exp(2\pi i \langle k,\gamma^{-1}(x,y)\rangle)
= exp(2\pi i \langle \gamma k,(x,y)\rangle)=\omega_{\gamma k}(x,y) 
\end{equation}
by the orthogonality of $\gamma$.
Then, we have:
\begin{proposition}[Dionne and Golubitsky \cite{digolu92}]\label{irredutivel}
The space $V = V_{k_{1}}\oplus\ldots\oplus V_{k_{s}}\subset \mathcal{\LL}$ is $\Gamma_{\LL}$-irreducible if and only if $\lbrace\pm k_{1},\ldots ,\pm k_{s}\rbrace$ is an $H_{\LL}$-orbit in $\LL^{*}$. In particular, $2s$ divides the order of $H_{\LL}$.
\end{proposition}

We say that $V$ is  \emph{generated by the orbit} of $k$.
Since $H_{\LL} \subset O(n+1)$, it follows that all the $k_{j}$ ($j=1,\hdots,s$) have the same norm.

\subsection{Projections of $\LL$-periodic Functions}\label{subsecProjections}

\begin{definition}\label{defProjection}
For \(f \in \XX_{\LL}$ and \(y_{0} > 0$, \emph{the projection operator} $\Pi_{y_{0}}$ is given by:
$$
\Pi_{y_{0}}(f)(x) = \int _{0}^{y_{0}}f(x,y)dy .
$$
The region $\{(x,y)\in\RR^{n+1}:\ 0\le y\le y_0\}$ is called the  \emph{band of projection}. 
\end{definition}	

Hence, we are interested in describing the effect of the projection $\Pi_{y_0}$ on $\gamaL$-irreducible subspaces of $\XX_\LL$.
%For this we need some more definitions.
%
%Given a subgroup $\Gamma$ of $E(n+1)$, the kernel of the homomorphism $\phi:  \Gamma  \rightarrow  O(n+1)$ given by $\phi(v,\delta)=\delta$ is  called the \emph{translation subgroup} of $\Gamma$. 
%A \emph{crystallographic group} is a discrete subgroup of $E(n+1)$, whose translation subgroup is an $(n+1)$-dimensional lattice $\LL$.	
%
%We consider the space, $\XX_{\Gamma}\subset\XX_\LL$, of $\Gamma$-invariant functions, given by
%$$
%\XX_{\Gamma} = \lbrace f:\RR^{n+1}\rightarrow\RR; \ \gamma\cdot f = f, \ \forall \gamma \in \Gamma \rbrace
%$$

Pinho and Labouriau \cite{labpin06,pinlab14} have characterised the group of symmetries of projected functions in the general case of functions invariant under a crystallographic group.
Applying  \cite[Proposition 3.1]{labpin06}, for $n=1$, or \cite[Theorem 1.2]{pinlab14}, for general $n$, to the special case where the crystallographic group only contains translations, it follows:
%From their results it follows:

\begin{proposition}\label{corollaryThPinho}
	All functions in $\Pi_{y_{0}}(\XX_\LL)$ are invariant under the action of the translation $v\in\RR^n$ if and only if one of the following conditions holds:
\begin{enumerate}
\renewcommand{\theenumi}{\Roman{enumi}}
\renewcommand{\labelenumi}{{\theenumi})}
\item $(v,0)\in \LL$;
\item $(0,y_{0})\in \LL$ and  $(v,y_{1})\in \LL$, for some $y_{1}\in\RR$.
\end{enumerate}
\end{proposition}

We denote by $\Pi_{y_0}(\LL) = \widetilde{\LL}$ the set of all translations under which the functions  in $\Pi_{y_{0}}(\XX_\LL)$ are invariant.

The action of $\Gamma_\LL$ on $\XX_\LL$ induces a similar action of a group on $\Pi_{y_{0}}(\XX_\LL)$.
Translations act as the  torus $\RR^n/\widetilde{\LL}$.
To see this, let $g=\Pi_{y_0}(f)$.
%=\int_0^{y_0}f(x,y)dy$
If $(v,0)\in\LL$ then $g(x-v)=\int_0^{y_0}f(x-v,y-0)dy\in\Pi_{y_0}(\XX_\LL)$.
If $(0,y_0)$ and $(v,y_1)\in\LL$, since $f(x,y+y_0)=f(x,y)$ then, 
$$
\int_0^{y_0}f(x-v,y-y_1)dy=\int_{y_1}^{y_0+y_1}f(x-v,z)dz=g(x-v)\in\Pi_{y_0}(\XX_\LL).
$$

Given $\alpha\in O(n)$, define $\alpha_\pm\in O(n+1)$ by:
$$
%\sigma := \left(\begin{array}{cc}
%I_{n} & 0 \\ 
%0 & -1
%\end{array}\right), \ 
\alpha_{+} := \left(\begin{array}{cc}
\alpha & 0 \\ 
0 & 1
\end{array}\right) \ \mbox{and} \ \  
\alpha_{-} := \left(\begin{array}{cc}
\alpha & 0 \\ 
0 & -1
\end{array}\right) .
%\alpha_{-} := \sigma\alpha_{+}.
$$
If either $\alpha_+$ or $\alpha_-\in\holoL$, then $\alpha$ belongs to the holohedry of $\widetilde{\LL}$.
For $g=\Pi_{y_0}(f)$ we get $\Pi_{y_0}(\alpha_+\cdot f)=\alpha \cdot g$.
If, in addition $(0,y_0)\in\LL$, then 
$$
\Pi_{y_0}(\alpha_-\cdot f)(x)=\int_0^{y_0} f(\alpha^{-1}x,- y)dy=\int_{-y_0}^{0} f(\alpha^{-1}x, z)dz=\alpha\cdot g(x)
$$
because we are integrating over a period.
We will denote by $\widetilde{J}=\Pi_{y_0}(\holoL)$ the group of  \emph{induced orthogonal symmetries}. 
It follows that $\alpha\in\widetilde{J}$ if and only if one of the conditions below holds:
\begin{enumerate}
\renewcommand{\theenumi}{\Roman{enumi}}
\renewcommand{\labelenumi}{{\theenumi})}
\item $\alpha_+\in\holoL$;
\item $(0,y_0)\in\LL$ and $\alpha_-\in\holoL$.
\end{enumerate}
%Thus, we will be interested in the action of $\Jup   = \lbrace \alpha_{+} , \alpha_-\in \holoL \rbrace$.

\subsection{Outline of the article and informal statement of results}

Given a lattice $\LL$ and its projection $\widetilde{\LL}=\Pi_{y_0}(\LL)$, the aim of this article is to relate functions in the space $\XX_\LL$ to their projections, that lie in the space $\XX_{\widetilde{\LL}}$, in order to obtain information about projected patterns and possible mode interactions.
The starting point is the following:
\begin{description}
\item[1st result]
Given a lattice $\LL$ with  $\widetilde{\LL}=\Pi_{y_0}(\LL)$, for almost all $y_0\in\RR$ we have $\Pi_{y_0}(\XX_\LL)=\XX_{\widetilde{\LL}}$.
\end{description}

This result is proved in Section~\ref{secProjectXL},  where we obtain explicitly necessary and sufficient conditions on $y_0$ under which 
$\Pi_{y_0}(\XX_\LL)=\XX_{\widetilde{\LL}}$. In such cases, all functions invariant under $\widetilde{\LL}$ can be obtained by the projection of a function invariant under $\LL$.
% the image of $\XX_\LL$ by the projection $\Pi_{y_0}$ is the whole function space $\XX_{\widetilde{\LL}}$.

Since solutions whose existence is guaranteed by the Equivariant Branching Lemma have less symmetry than the whole group, the relation between the symmetry of a function and its projection helps clarify what possible symmetries appear in the projected observation. When the symmetry of a function has a correspondence in the projected lattice, we know which possible symmetries to expectin the observed projected solutions. The Equivariant Branching Lemma then states that generically the solutions with maximal projected symmetry are observed.

The main result concerns mode interactions:
\begin{description}
\item[2nd result]
Given a lattice $\LL$ with  $\widetilde{\LL}=\Pi_{y_0}(\LL)$, generically and
for almost all $y_0\in\RR$ the projection of a single mode is a mode interaction.
\end{description}

Explicit conditions for the genericity and restrictions on the values of $y_0$ are given in Section~\ref{secIrreducible}, where this result is proved.
This shows the compatibility of the explanations provided for the black-eye pattern by \cite{gunaouswi94} and \cite{gomes99}: when the outcome of an experiment is observed through a projection, generically, what is originally a single mode solution is projected as an observed mode interaction. Hence a 3-dimensional single mode (as explained by \cite{gomes99}) may be perceived as a mode interaction (as interpreted by \cite{gunaouswi94}).

%The main result appears in Section~\ref{secIrreducible}, where we obtain conditions under which a single mode bifurcation in $\XX_\LL$  is projected into a mode interaction in $\XX_{\widetilde{\LL}}$, and show that this situation is rather common.
Some  algebraic results that are used in the proof are developed in 
Section~\ref{secL*}.
%that contains 
These results concern a description of the action of $\holoL$ on the dual lattice $\LL^*$ and its behaviour under projection. 
The results of Sections~\ref{secL*} and ~\ref{secIrreducible} are illustrated by the projection from the simple cubic lattice in $\RR^3$, in two different positions,
as well as the projection of the body centred cubic lattice, in a position that creates black-eyes.

\section{Projection of the space of $\LL$-periodic functions}\label{secProjectXL}
Given an $(n+1)$-dimensional lattice $\LL$ and its projection $\widetilde{\LL}=\Pi_{y_0}(\LL)$,  the projection of the space $\XX_\LL$ of all $\LL$-periodic functions does not necessarily coincide with the space of all functions with period in $\widetilde{\LL}$.
In this section we  characterise the situation when this is true.
We illustrate this with examples, the end of each one marked {\Large$\bullet$}.

Denote by $P:\RR^{n+1}\rightarrow\RR^{n}$  the  projection $P(x,y) = x$ and
 by $\lbrace  y = 0 \rbrace$ the space $\lbrace (x,y)\in \RR^{n+1}; \ y = 0\rbrace$.

\begin{lemma}\label{lemaLstar}
Let $\LL$ be an (n+1)-dimensional lattice with $\widetilde{\LL}=\Pi_{y_0}(\LL)$.
Then $\widetilde{\LL}^{*}\subset P(\LL^{*})$.
\end{lemma}

\begin{proof}
We analyse the two cases of Proposition~\ref{corollaryThPinho} . 
If  $(0,y_0)\in\LL$,  then $\widetilde{\LL}= P(\LL)$. 
If $\tilde{k}\in \widetilde{\LL}^*$ and $(v,z)\in\LL$ then  $\langle\tilde{k},v\rangle=\langle(\tilde{k},z),(v,0)\rangle\in\ZZ$ because $v\in\widetilde{\LL}$ and hence $(\tilde{k},0)\in\LL^*$, therefore $\tilde{k}\in P(\LL^*)$, as we wanted.

If $(0,y_0)\not\in\LL$ then $\widetilde{\LL}= P(\LL\cap \lbrace  y = 0 \rbrace)$. 
If $(\tilde{k},z)\in\LL^*$ then for every $v\in\widetilde{\LL}$ we have  $(v,0)\in\LL$ and thus  $\langle(\tilde{k},z),(v,0)\rangle=\langle\tilde{k},v\rangle\in\ZZ$, hence $\tilde{k}\in\widetilde{\LL}^*$, therefore $\widetilde{\LL}^*=P(\LL^*)$.
\end{proof}
		
\begin{theorem}\label{eq_periodic_functions}
	Let $\LL\subset \RR^{n+1}$ be  an (n+1)-dimensional  lattice and $\widetilde{\LL} = \Pi_{y_{0}}(\LL)$. The space $\Pi_{y_{0}}(\XX_{\LL})$ is equal to $\XX_{\widetilde{\LL}}$ if and only if for each $\tilde{k} \in \widetilde{\LL}^{*}$ there exists $z \in \RR$ such that  $(\tilde{k}, z) \in \LL^{*}$ and $zy_{0} \notin \ZZ\setminus\lbrace 0\rbrace$.
\end{theorem}
\begin{proof}
	Suppose that $\Pi_{y_{0}}(\XX_{\LL}) = \XX_{\widetilde{\LL}}$. 
	Then, for all $\omega_{\tilde{k}} \in \XX_{\widetilde{\LL}}$, there exists a non-zero function $f \in \XX_{\LL}$ such that
$$
\omega_{\tilde{k}}(x) = \Pi_{y_{0}}(f)(x) \neq 0.
$$

By Lemma~\ref{lemaLstar} we know that $\widetilde{\LL}\subset P(\LL)$.
Since $f \in \XX_{\LL}$ admits a Fourier expansion in terms of the waves $\omega_{k}$ then, without loss of generality we may take $f = c\ \omega_{(\tilde{k},z)}$, for some $(\tilde{k},z) \in \LL^{*}$. Moreover since $\Pi_{y_{0}}(f)(x) \neq 0$, then $\int_{0}^{y_{0}}w_{z}(y)dy \neq 0$. Therefore, $zy_{0} \notin \ZZ\setminus\lbrace 0\rbrace$.
	
	Conversely,  it is clear from Proposition~\ref{corollaryThPinho}  that $\Pi_{y_{0}}(\XX_\LL) \subseteq\XX_{\widetilde{\LL}}$.  To get the other inclusion, we want to show that all the wave functions  $\omega_{\tilde{k}} \in \XX_{\widetilde{\LL}}$ are in $\Pi_{y_{0}}(\XX_\LL)$. 
Suppose that for each $\tilde{k} \in \widetilde{\LL}^{*}$ there exists $z \in \RR$ such that  $(\tilde{k}, z) \in \LL^{*}$ and $zy_{0} \notin \ZZ\setminus\lbrace 0\rbrace$. 	 
Then,
$$
\Pi_{y_{0}}(\omega_{(\tilde{k}, z)}) = c_{0}\omega_{\tilde{k}}, \ \ \text{with} \ \ c_{0} = \int_{0}^{y_{0}}w_{z}(y)dy.
$$
Since $zy_{0} \notin \ZZ\setminus\lbrace 0\rbrace$, then $c_{0} \neq 0$. Therefore $\Pi_{y_{0}}(\XX_\LL) = \XX_{\widetilde{\LL}}$. 
\end{proof}			

From Theorem~\ref{eq_periodic_functions} it follows:	 
\begin{corollary} 
	Let $\LL\subset \RR^{n+1}$ be  an (n+1)-dimensional  lattice and $\widetilde{\LL} = \Pi_{y_{0}}(\LL)$.
%	 If $\widetilde{\LL}_s$ is rationally compatible with $\LL$, t
	 Then for almost all values of $y_0$ we have  $\Pi_{y_{0}}(\XX_{\LL}) = \XX_{\widetilde{\LL}}$.
\end{corollary}

\begin{proof}
%	Suppose that $\widetilde{\LL}_s$ is rationally compatible with $\LL$, then by Proposition \ref{dual_rationality}, 
	By Lemma~\ref{lemaLstar}, for each $\tilde{k} \in \widetilde{\LL}^{*}$ there exists $z \in \RR$ such that  $(\tilde{k}, z) \in \LL^{*}$. 
Let $(k_1,k_2) \in \RR^n\times\RR$ and consider the discrete set 
$$
K =\bigcup_{\substack{k_{2}\neq 0; \\ (k_{1}, k_{2}) \in \LL^{*}}}\ZZ\cdot \frac{1}{k_{2}}
$$
%K =\displaystyle\bigcup_{\substack{k_{2}\neq 0; \\ (k_{1}, k_{2}) \in \LL^{*}}}\ZZ\cdot \frac{1}{k_{2}}$, 
then the complement of $K$ is a dense set in $\RR$. 
	
For $y_{0} \in \RR\setminus K$, we have $\int_{0}^{y_{0}}w_{k_{2}}(y)dy \neq 0$. This implies that, for all $y_{0} \notin K$, we have  $\Pi_{y_{0}}(\XX_{\LL}) = \XX_{\widetilde{\LL}}$.
\end{proof}

\section{Projection of  $\holoL$-orbits in $\LL^*$}\label{secL*}
We want to describe the effect of $\Pi_{y_0}$ on irreducible subspaces of $\XX_\LL$.
Proposition~\ref{irredutivel} tells us that this is done by studying  the orbit of elements $k$ of $\LL^*$ under the holohedry $\holoL$. 
The projection of the irreducible subspace generated by this orbit can be decomposed into irreducible subspaces under the action of symmetries of $\widetilde{\LL}=\Pi_{y_0}(\LL)$ that lie in the group  $\widetilde{J} =\Pi_{y_0}(\holoL)$ that was defined in Subsection~\ref{subsecProjections} above.
This is equivalent to  decomposing $P(\holoL\cdot k)$ into $\widetilde{J} $-orbits in $\widetilde{\LL}^*$.
 To do this, we decompose $\holoL $ into subsets such that the orbit of $k$ under each of these subsets is projected into exactly one $\widetilde{J} $-orbit.
	
	Define the subset $\Jup  $ of $\holoL $ as
\begin{equation}\label{J_til_suspendido}
\Jup   = \lbrace \alpha_{+} \in \holoL ; \ \alpha \in \widetilde{J} \rbrace\cup \lbrace \alpha_{-} \in \holoL ; \ \alpha \in \widetilde{J} \rbrace.
\end{equation}
Then for every $k = (\tilde{k},z) \in \LL^{*}$,
the projection $P(\Jup\cdot  k)\subset \widetilde{J}\cdot \tilde{k}$. However,  the set of $\delta \in \holoL $ such that $P(\delta k) \in \widetilde{J}\cdot  \tilde{k}$ is in general larger than $\Jup  $. Next, we describe the subset of $\delta\in\holoL $ such that the projection by $P$ of $\delta k$ lies in $\widetilde{J}\cdot \tilde{k}$.
	
%	By checking the group axioms, it is easy to see that the set $\Jup  $ is a subgroup, but not necessarily a normal subgroup of $J$. To see this, consider the subgroup $\widehat{\Gamma}$ of $\Gamma$, whose elements are of the form
%\begin{equation}\label{Gamma_chapeu}
%\left((v,y),\alpha_\pm\right);\ \alpha \in O(n), \ (v,y) \in \RR^{n}\times\RR.
%\end{equation}
%See\cite{pinho06, Olicalab15}. As we can see, the set $\Jup  $ is a subgroup of the point group $\Jup $ of $\widehat{\Gamma}$. In particular, when the projection band satisfies $(0,y_{0}) \in \LL$, then $\Jup   = \Jup $.
	
	Given $k = (\tilde{k},z) \in \LL^{*}$ and $\alpha \in \widetilde{J} $, it is convenient to define $J^{\alpha}_{k}$, the subset of $\holoL $ given by:
\[J^{\alpha}_{k} = \lbrace \delta \in \holoL ; \ P(\delta(\tilde{k},z))=\alpha \tilde{k}\rbrace.\]
Denote the union of the sets $J^{\alpha}_{k}$ by
\[S_{k}  = \bigcup_{\alpha \in \widetilde{J} }J^{\alpha}_{k}.\]
Then the projection of $S_k\cdot k$ satisfies $P(S_{k}\cdot k) = \widetilde{J} \cdot P(k)$.
%$P(S_{k}k) \subseteq \widetilde{J} P(k)$.	
We show in Theorem \ref{sepa} that $\holoL k$ is decomposed into orbits given by   cosets $\delta S_{k}$ and that  $P(\delta S_{k}\cdot  k) = \widetilde{J}\cdot  P(\delta k)$.	
The main problem is that $S_k$ is not necessarily a group.
We start with some properties of the sets we have defined.
From now on for $v=(x,y)\in\RR^{n+1}$, $x\in\RR^n$, $y\in\RR$ we use the notation $v_{|2}=y$.

\begin{proposition}\label{caracterizacao}	
	The following properties hold for $k = (\tilde{k},z) \in \LL^{*}$:
\begin{enumerate}
\renewcommand{\theenumi}{\roman{enumi}}
\renewcommand{\labelenumi}{({\theenumi})}
\item\label{p61}
 If $\delta \in J^{Id_{n}}_{k}$, then the last coordinate, $\delta  k_{|2}$, of $\delta k$ is equal to  $\pm z$. 

\item\label{p62}
  For $\alpha \in \widetilde{J} $  then $J^{\alpha}_{k}=\gamma \cdot J^{Id_{n}}_{k}$
where $\gamma\in\holoL$ is either $\gamma = \alpha_{+}$ or $\gamma=\alpha_{-}$. 

\item\label{p63}
  Let $\Sigma_{k}$ be the isotropy subgroup of $k$ in $\holoL $. Then either $J^{Id_{n}}_{k} = \Sigma_{k}$ or it is the disjoint union $\Sigma_{k}\cupdot \beta_{-}\Sigma_{k}$, for some $\beta_{-} \in J^{Id_{n}}_{k}$.

\item\label{p64}
  The set $S_{k}$ satisfies $S_{k} = \displaystyle\bigcup_{\alpha \in \widetilde{J} }J^{\alpha}_{k} = \Jup   J^{Id_{n}}_{k}$. 
  Moreover, $S_{k}\cdot k = \Jup\cdot  k$.
  
\item\label{p65} If $J_k^{Id_n}$ is a subgroup of $\holoL$ and if $\Jup   J^{Id_{n}}_{k} = J^{Id_{n}}_{k}\Jup $, then
  $S_{k}$ is a subgroup of $\holoL $.

%\item\label{p65}
%  $S_{k}$ is a subgroup of $\holoL $ if and only if $\Jup   J^{Id_{n}}_{k} = J^{Id_{n}}_{k}\Jup  $.
\end{enumerate}	
\end{proposition}

\begin{figure}
\includegraphics[scale=0.5]{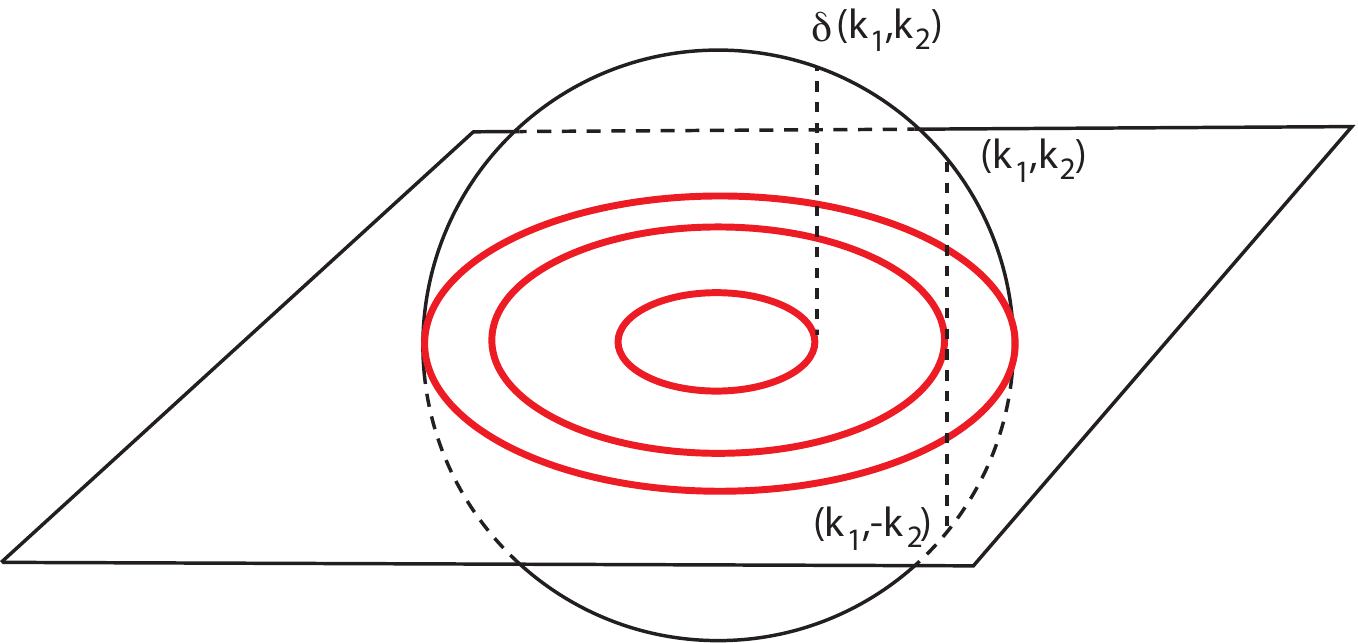}
\caption{If $P(\delta (k_1,k_2))=k_1$, then  $\delta (k_1,k_2)_{|_2}=\pm k_2$. 
Conversely, if $\delta (k_1,k_2)_{|_2}\ne \pm k_2$ then $k_1$ and $P(\delta (k_1,k_2)$ do not have the same norm and thus cannot be in the same $\widetilde{J}$-orbit.}
\label{figJorbitas}
\end{figure}

\begin{proof}
	 For $\delta \in J^{Id_{n}}_{k}$, we write $\delta k = (\tilde{k},(\delta k)_{|2})$.
Then by orthogonality of $\holoL $: 
\[\Vert (\tilde{k},z)\Vert = \Vert\delta k\Vert = \Vert (\tilde{k},\delta(\tilde{k},z)_{|2}\Vert\]	 
Therefore $\vert(\delta k)_{|2}\vert = \vert z\vert$. This proves item~\eqref{p61}.
See Figure~\ref{figJorbitas} for an illustration of this property.

	To prove item~\eqref{p62}, let $\gamma = \alpha_{+}$ or $\gamma = \alpha_{-}$, depending on whether either $\alpha_{+}$ or $\alpha_{-}$ is in $\holoL $. 
	Let $\phi:  J^{Id_{n}}_{k}\rightarrow J^{\alpha}_{k}$ given by $ \delta\mapsto \phi(\delta) = \gamma\delta $.
We show that the map $\phi$ is injective and onto.
	
	In fact, if $\phi(\delta_{1}) = \phi(\delta_{2})$, for some $\delta_{1}, \ \delta_{2} \in J^{Id_{n}}_{k}$, then $\gamma\delta_{1} = \gamma\delta_{2}$, for $\gamma = \alpha_{+}$ or $\gamma = \alpha_{-}$, implying that $\delta_{1} = \delta_{2}$. Thus $\phi$ is injective.
	
	Now consider $\rho \in J^{\alpha}_{k}$, then there exist $\alpha_{\pm}^{-1}\rho \in \holoL $ and $\tilde{k}_{2}$, such that 
\[\alpha_{\pm}^{-1}\rho(\tilde{k},z) = \alpha_{\pm}^{-1}(\alpha\tilde{k},\tilde{k}_{2}) = (\tilde{k},\pm\tilde{k}_{2})\]		
that is, $\alpha_{\pm}^{-1}\rho \in J^{Id_{n}}_{\tilde{k}}$ and $\phi(\alpha_{\pm}^{-1}\rho) = \rho$.  	
	
	The items~\eqref{p63} and \eqref{p65} are now immediate.
\end{proof}
Note that  item~\eqref{p63} implies that,  in general, if $J_k^{Id_n}\ne\Sigma_k$ then $J_k^{Id_n}$ is not a group, as the next example will show. 
%If   $\Jup   J^{Id_{n}}_{k} \neq J^{Id_{n}}_{k}\Jup  $ then $S_{k}$ is also not a subgroup of $\holoL $.
	
\begin{example}\label{ex1}
Let $\LL_{1}$ be the simple 3-dimensional  cubic lattice  generated by  $(1,0,0), \ (0,1,0), \ (0,0,1)$, with dual lattice $\LL_{1}^{*}$ generated by the same vectors.

The holohedry $H_{\LL_{1}}$ has 24 rotations, they are generated by: the identity $Id_{3}$; three rotations of order 4: $R_{x}$: rotation about $(1,0,0)$; $R_{y}$: rotation about $(0,1,0)$ and $R_{z}$: rotation about $(0,0,1)$. The remaining rotations are denoted by $R_{v}$  as a rotation about the axis $v$: $R_{(1,1,1)}$, $R_{(1,-1,-1)}$, $R_{(1,-1,1)}$, $R_{(1,1,1)}$ of order 3, and  $R_{(1,0,1)}$,  $R_{(1,0,-1)}$, $R_{(1,1,0)}$, $R_{(1,-1,0)}$, $R_{(0,1,1)}$  and $R_{(0,1,-1)}$of order 2.

The remaining elements of $H_{\LL_{1}}$ can be obtained by multiplying these matrices by $-Id_{3}$.
%\hfill {\Large$\bullet$}
%\end{example}

\begin{table}[hhh]
\caption{Relation of the set $S_{k}$ for all $k\in\LL_1^*\setminus\{(0,0,0)\}$ and the group $\holoL$ for Example \ref{ex1}. 
Note that 
$\vert S_{(a,b,b)}\vert = 32$,
 hence  $S_{(a,b,b)}$ is not a subgroup of $\holoL $, since $\vert\holoL\vert=48$. } 
\label{table:S_set}
\begin{center}
\begin{tabular}{ l|l|l|l}
%\hline
$k$ with $abc\neq 0$& $\Sigma_{k}$& $J^{Id_2}_{k}$ &$S_{k}$  \\ 
$a \neq b \neq c\neq a$&&\\
\hline \hline
&&&\\
%$a \neq 0$ and $b = c=0$&
$(a,0,0)$&
$\Sigma_{(a,0,0)}=J^{Id_{2}}_{(a,0,0)}$
& $\begin{aligned} J^{Id_{2}}_{(a,0,0)} & = \lbrace Id_{3}, R_{x}, R_{x}^{2}, R_{x}^{3}, \\
                                       &-R_{y}^{2},-R_z^2, -R_{(0,1,1)},-R_{(0,1,-1)}\rbrace
\end{aligned}$
& $\vert S_{(a,0,0)}\vert = 32$
%, then  $S_{(a,0,0)}\nleq \holoL $ 
\\
&&\\
\hline
&&\\
$(a,b,b)$&
%$0\neq a \neq b = c\neq 0$&
$\Sigma_{(a,b,b)}=\lbrace Id_3, -R_{(0,1,-1)}\rbrace$
& $ J^{Id_{2}}_{(a,b,b)}  = \lbrace  Id_3,-R_z^2,R_x^3, -R_{(0,1,-1)}\rbrace$
& $\vert S_{(a,b,b)}\vert = 32$
%, then  $S_{(a,0,0)}\nleq \holoL $ 
\\
&&\\
\hline
&&\\
$(a,a,b)$&
%$0\neq a=b \neq c\neq 0$&
$\Sigma_{(a,a,b)}=\lbrace Id_3, -R_{(1,-1,0)}\rbrace$
& $ J^{Id_{2}}_{(a,a,b)}  = \lbrace  Id_3,-R_z^2, R_{(1,1,0)},-R_{(1,-1,0)}\rbrace$
& $S_{(a,a,b)}=\Jup $
%, then  $S_{(a,0,0)}\nleq \holoL $ 
\\
&&\\
\hline
&&\\
$(a,a,0)$&
%$a = b \neq 0$ and $c=0$&
$\Sigma_{(a,a,0)}=J^{Id_{2}}_{(a,a,0)}$
&  $J^{Id_{2}}_{(a,a,0)} = \lbrace Id_{3}, -R_z^2, R_{(1,1,0)},-R_{(1,-1,0)}\rbrace $
& $S_{(a,a,0)} = \Jup $ \\ 
&&\\
\hline  
&&\\
$(a,b,0)$&
%$\begin{aligned}a \neq 0 \ne& b\quad &\\
%\mbox{ and }\  a \neq& b\end{aligned}$&
$\Sigma_{(a,b,0)}=J^{Id_{2}}_{(a,b,0)}$
& $J^{Id_{2}}_{(a,b,0)} = \lbrace Id_{3},  -R_z^2\rbrace$
& $S_{(a,b,0)} = \Jup $\\
&&\\
\hline  
&&\\
$(a,b,c)$&
%$\begin{aligned}a \neq b \neq c\neq a\\ \mbox{and }\  abc \neq &0\end{aligned}$&
$\Sigma_{(a,b,c)}= \lbrace Id_{3}\rbrace$
& $J^{Id_{2}}_{(a,b,c)} = \lbrace Id_{3},  -R_z^2\rbrace$
& $S_{(a,b,c)} = \Jup $\\
&&\\
\hline  
&&\\
$(a,a,a)$&
%$a=b = c\neq 0$&
$\begin{aligned}\Sigma_{(a,a,a)}&= \lbrace Id_{3}, R_{(1,1,1)},\\
				& R_{(1,1,1)}^{2},  -R_{(1,0,-1)},\\
				& -R_{(0,1,-1)},-R_{(1,-1,0)}\rbrace
\end{aligned}$
& $\begin{aligned} J^{Id_{2}}_{(a,a,a)} & = \Sigma_{(a,a,a)}\cup \lbrace  R_{x}^3, R_{y}, -R_{z}^{2}, \\
                                       &-R_{(1,-1,-1)}, -R_{(1,-1,1)}^2,-R_{(1,1,0)}\rbrace
\end{aligned}$
& $S_{(a,a,a)}=H_{\LL_1}$
 \end{tabular}  
\end{center}    
\end{table}

For any $y_0$, the projected lattice $\Pi_{y_{0}}(\LL_1) = \widetilde{\LL}_1$ is generated by the vectors $(1,0), \ (0,1)$, its dual, $\widetilde{\LL}_1^{*}$ is generated by the same vectors.

The subgroup $\Jup $ has order 16 and is given by
$\{\pm Id_3,\pm R_x^2, \pm R_y^2, \pm R_z,\pm R_z^2,\pm R_z^3, \pm R_{(1,1,0)},\pm R_{(1,-1,0)}\}$.

The orbit of an element $(a,b,c) \in \LL_1^*$ by $H_{\LL_1}$ is given by:
\begin{equation}\label{J(a,b,0)}
\lbrace(\pm a,\pm b,\pm c), (\pm b, \pm a, \pm c)\rbrace 	
  			   \cup \lbrace (\pm a, \pm c, \pm b), (\pm c, \pm a, \pm b)\rbrace 
  			   \cup \lbrace (\pm b, \pm c, \pm a), (\pm c, \pm b, \pm a)\rbrace
\end{equation}
and  the projection $P(H_{\LL_1}(a,b,c))$ is given by:
$$
\lbrace (\pm a,\pm b), (\pm b,\pm a)\rbrace 
			 \cup \lbrace (\pm a,\pm c), \ (\pm c,\pm a)\rbrace 
			 \cup \lbrace (\pm c,\pm b), \ (\pm b,\pm c)\rbrace
$$

On the other hand, for any $y_0\in\RR$, the group $\widetilde{J} =\Pi_{y_0}(H_{\LL_1})$ is the dihedral group of symmetries of the square, $D_{4}$,  generated by:
$$
\gamma = \left(\begin{array}{cc} 
0 & 1  \\ 
-1 & 0 \end{array}\right),
\qquad
\kappa = \left(\begin{array}{cc} 
0 & 1 \\ 
1 & 0 
\end{array}\right) .
$$
For every $\alpha\in\widetilde{J}$ both $\alpha_+$ and $\alpha_-\in\holoL$.
 The orbit $\widetilde{J} (a,b)$ is $\lbrace (\pm a,\pm b), \ (\pm b,\pm a)\rbrace$.

%The subgroup $\Jup $ is given by $\lbrace \pm Id_{3}, \ \pm R_{x}^{2}, \ \pm R_{y}^{2}, \ \pm R_{z}, \ \pm R_{z}^{2}, \ \pm R_{z}^{3}, \ \pm R_{(1,1,0)}, \ \pm R_{(1,-1,0)}\rbrace$.  
%As we can see, $\Jup  = \Jup  $.

%Note that, with the exception of $k=(a,b,b)$, for all other $k \in \LL_1^{*}$, we have  $J^{Id_{2}}_{k} = \Sigma_{k}$. Moreover, by Proposition \ref{caracterizacao} item 2, for each $\alpha \in \widetilde{J} $, the set $J^{\alpha}_{k} = \alpha_{\pm} J^{I_{2}}_{k}$, depending on whether either $\alpha_{+}$ or $\alpha_{-}$ is in $\holoL $.

In Table \ref{table:S_set} we describe, for each $k \in \LL_1^{*}\setminus \{(0,0,0)\}$, the subgroup $\Sigma_k$ and the set $J^{Id_{2}}_{k}$ as well as the way  the set $S_{k}$ and the group $\holoL $ are related. 
Some special cases are worth mentioning.
If $b\neq 0 \neq a$ then $J_{(a,b,b)}^{Id_2}$ is not a group, because it contains the order 4 element $R_x^3$ but not its powers.
In the other cases where $J_k^{Id_2}\neq \Sigma_k$ we have that  $\beta_-$ has  order 2 and $J_k^{Id_2}$ is a group.

The only cases where $S_k\ne \Jup $ are for $k=(a,a,a)$, where $S_k$ is the whole group $H_{\LL_1}$ and for $k=(a,b,b)$, $0\ne a\ne b$.
For the second case, note that $\delta\in\holoL$ is not in $S_{(a,b,b)}$ if $\delta(a,b,b)=(\pm b, \pm b,\pm a)$.
There are 8  rotations with this property: $R_y$ and $R^3_y$; the order 2 rotations $R_{(1,0,1)}$ and $R_{(1,0,-1)}$; and one and only one rotation of order 3 around each of the four axes of order 3.
Multiplying them by $-Id_{3}$ yelds 16 elements not in $S_{(a,b,b)}$, hence $\vert S_{(a,b,b)}\vert=32$ and thus $S_{(a,b,b)}$ is not a subgroup of $H_{\LL_1}$.
% $\holoL$.
 \hfill {\Large$\bullet$}
\end{example}

\begin{example}\label{ex2}
Consider now the lattice $\LL_2$ that has generators $ \left(1,0,0\right), \left(\frac{1}{2}, \frac{\sqrt{3}}{2},0\right), \left(\frac{1}{2},\frac{\sqrt{3}}{6},\frac{\sqrt{6}}{6}\right)$ given by $\LL_{2} =\frac{1}{\sqrt{2}} A\LL_{1}$ where  $A$ is:
%Consider now a crystallographic group $\Gamma_{2}$ obtained from $\Gamma_{1}$, given in the previous example, by a change of coordinate given by the matrix
$$
A =
\left(\begin{array}{ccc} 
\frac{1}{\sqrt{2}} & -\frac{1}{\sqrt{2}} & 0 \\ 
\frac{1}{\sqrt{6}} & \frac{1}{\sqrt{6}} & \frac{2}{\sqrt{6}} \\
\frac{1}{\sqrt{3}} & \frac{1}{\sqrt{3}} & -\frac{1}{\sqrt{3}}\end{array}\right).
$$
The lattice  $\LL_2$ has holohedry $H_{\LL_{2}} = AH_{\LL_{1}}A^{-1}$.  

If $y_0\ne n\sqrt{6}/2$ with $n\in\ZZ\setminus\{0\}$, then  $(0,0,y_0)\not\in\LL_2$.
In this case $\widetilde{\LL}_2$ is generated by $(1,0)$ and $(\frac{1}{2},\frac{\sqrt{3}}{2})$
and $\widetilde{J}$ is isomorphic to the dihedral group $D_3$. 
However, the holohedry of $\widetilde{\LL}_2$ is larger than $\widetilde{J}$.
It contains rotations of order 6 and is isomorphic to $D_6$.

For $y_0= n\sqrt{6}/2$ with $n\in\ZZ\setminus\{0\}$, then $\widetilde{\LL}_2$ is generated by $(1,0)$ and $(\frac{1}{2},\frac{\sqrt{3}}{6})$
and $\widetilde{J}$ is isomorphic to the dihedral group $D_6$, coinciding with the holohedry of $\widetilde{\LL}_2$.

The dual lattice $\LL_2^*=\sqrt{2} A\LL_1^*$ is generated by $(0,0,\sqrt{6})$, $(1,-\frac{\sqrt{3}}{3},-\frac{\sqrt{6}}{3})$, and 
$(0,\frac{2\sqrt{3}}{3},-\frac{\sqrt{6}}{3})$.

For $k= (2,0,0) \in \LL_{2}^{*}$, the set $J_{(2,0,0)}^{Id_2}=\Sigma_{(2,0,0)}$ is a subgroup of $H_{\LL_2}$.
%contains, besides the identity, the following elements of order 2, forming a subgroup of $H_{\LL_2}$:
%$$
% AR_{(1,-1,0)}A^{-1}=
%\left(\begin{array}{ccc} 
%1&0&0\\
%0&-1&0\\
%0&0&-1
%\end{array}\right)
%\qquad
%-AR_z^2A^{-1}=
%\left(\begin{array}{ccc} 
%1&0&0\\
%0&\frac{1}{3}&\frac{2\sqrt{2}}{3}\\
%0&\frac{2\sqrt{2}}{3}&\frac{1}{3}
%\end{array}\right)
%\qquad 
%-AR_{(1,1,0)}A^{-1}=
%\left(\begin{array}{ccc} 
%1&0&0\\
%0&\frac{1}{3}&-\frac{2\sqrt{2}}{3}\\
%0&\frac{-2\sqrt{2}}{3}&-\frac{1}{3}
%\end{array}\right).
%$$
The $\widetilde{J}$-orbit of $P(2,0,0)=(2,0)$ contains 6 elements, independently of $y_0$, and is
$$
\widetilde{J}\cdot (2,0)=\{ (\pm 2,0), (\pm 1,\pm \sqrt{3})\}.
$$
We claim that  $S_{(2,0,0)}$ is not a group.
To see this we compute
$$
AR_yA^{-1}=
\left(\begin{array}{ccc} 
\frac{1}{2}&\frac{\sqrt{3}}{6}&-\frac{2\sqrt{3}}{3}\\
-\frac{\sqrt{3}}{2}&\frac{1}{6}&-\frac{\sqrt{2}}{3}\\
0&\frac{2\sqrt{2}}{3}&\frac{1}{3}
\end{array}\right)
\qquad 
AR_y^2A^{-1}=
\left(\begin{array}{ccc} 
0&-\frac{\sqrt{3}}{3}&-\frac{\sqrt{6}}{3}\\
-\frac{\sqrt{3}}{3}&-\frac{2}{3}&\frac{\sqrt{2}}{3}\\
-\frac{\sqrt{6}}{3}&\frac{\sqrt{2}}{3}&-\frac{1}{3}
\end{array}\right)
\qquad 
AR_y^3A^{-1}=
\left(\begin{array}{ccc} 
\frac{1}{2}&-\frac{\sqrt{3}}{2}&0\\
\frac{\sqrt{3}}{6}&\frac{1}{6}&\frac{2\sqrt{2}}{3}\\
-\frac{\sqrt{6}}{3}&-\frac{\sqrt{2}}{3}&\frac{1}{3}
\end{array}\right) .
$$
Then $AR_{y}A^{-1} \in S_{(2,0,0)}$, since $P(AR_{y}A^{-1}(2,0,0))=(1,-\sqrt{3})$ and yet the powers of $AR_{y}A^{-1} $ are not in $S_{(2,0,0)}$,
establishing the claim.
\hfill {\Large$\bullet$}
%By doing some calculations, we can see that for the vector $(\tilde{k}, z) = (2,0,0) \in \LL_{2}^{*}$,
%the product $S_{(2,0,0)} = \Jup   J^{Id_{n}}_{(2,0,0)}$ is not a subgroup of $H_{\LL_2}$, since the element $AR_{x}A^{-1} \in S_{(2,0,0)}$ does not have an inverse in $S_{(2,0,0)}$.
\end{example}
\bigbreak

	 We can now prove the main result of this section. See Figure~\ref{figJorbitas}.
	
\begin{theorem}\label{sepa}
	Consider an $(n+1)$-dimensional lattice $\LL$ with holohedry $\holoL $, and let $\widetilde{J}  = \Pi_{y_{0}}(\holoL)$. Let $k  \in \LL^{*}$.
%	$k = (\tilde{k},z) \in \LL^{*}$. 
%	Given $\delta \in \holoL $, the $\widetilde{J} $-orbit of $P(\delta k)$ is $P(S_{\delta k}\delta k)$.  
Then there are $\delta_j \in \holoL$, $j-2, \hdots, r$ yielding $k_1=k,k_2=\delta_2k, \ldots, k_r=\delta_rk$ such that 
$\holoL\cdot k$ is the disjoint union
$$
\holoL\cdot k=(S_{k_1}\cdot k_1)\cup (S_{k_2}\cdot k_2)\cup\ldots\cup (S_{k_r}\cdot k_r)
$$
and therefore the projection $P(\holoL k)$ is a disjoint union of $\widetilde{J} $-orbits.
	%, all of which have the same number of elements, counted with multiplicity.
\end{theorem}	
\begin{proof}
%	The first statement follows by definition.
Given $\alpha\in\widetilde{J}$ let $\alpha_*\in\Jup\subset\holoL$ stand for either $\alpha_+$ or $\alpha_-$, as appropriate. 
We claim that for any $\delta\in\holoL$, any $\alpha\in\widetilde{J}$ and $k\in\LL^*$, if $\alpha_*\delta\in S_k$ then $\delta \in S_k$.

If $\alpha_*\delta\in S_k$ then there exists $\beta\in\widetilde{J}$ such that $\alpha_*\delta k=\beta_* k$, by items \eqref{p61} and \eqref{p62} of Proposition~\ref{caracterizacao}. 
Hence, $\delta k=\alpha_*^{-1}\beta_* k$ and since $\alpha_*^{-1}\beta_*\in\Jup$ this implies that $\delta\in S_k$, as claimed.

	Now consider %$\holoL  = S_{k}\cupdot S_{k}^{c}$
	$\holoL  = S_{k}\cup S_{k}^{c}$ and $\delta_{2} \in S_{k}^{c}$, where $S_{k}^{c}$ is the complement of $S_{k}$. 
	Then for any $\alpha_*\in\Jup$ we have $\alpha_*\delta_2\in S_k^c$. 
	In particular, $\Jup\cdot (\delta_2 k)\cap S_k\cdot k=\emptyset$.
	On the other hand, by Proposition \ref{caracterizacao} item~\eqref{p64}, the set  $\Jup \cdot (\delta_2 k) = S_{\delta_2k}\cdot (\delta_2k) \subset \holoL k$. Thus, we can write
$$
 \holoL \cdot k = S_{k}\cdot k \cup S_{\delta_2k}\cdot (\delta_2k)\cup (S_{k}\cup S_{\delta_2k}\delta_{1})^{c}k.
 $$	  

	Since $\holoL $ is a finite group, we can repeat the process to obtain:
	\begin{equation}\label{J_orbit}
\holoL\cdot k  = S_{k}\cdot k \cup S_{\delta_2k}\cdot (\delta_2k) \cup \ldots \cup S_{\delta_{r}k}\cdot (\delta_{r}k).
\end{equation}
%\begin{equation}\label{J_orbit}
%\holoL\cdot k  = S_{k}\cdot k \cup S_{\delta_2k}\cdot (\delta_2k) \cup \ldots \cup S_{\delta_{r}k}\cdot (\delta_{r}k).
%\end{equation}
Then, for 
%$k = (\tilde{k},z)$, 
$u_1= P(k)$ and $= u_{i}P(\delta_{i}k) $, we have the disjoint union
\begin{equation}\label{wavelenths_1}
P(\holoL \cdot k) = \displaystyle \bigcup_{i = 1}^{r}\widetilde{J} \cdot u_{i}
\end{equation}	
and the result follows.
\end{proof}
We are interested in the case when there are several components  in \eqref{wavelenths_1}.
The next corollary provides a condition for this, in terms of the size of the different sets that are used in this section.

\begin{corollary}\label{moreProjectedOrbits}
For $k\in\LL^*$, the projection $P(\holoL\cdot k)$  in   \eqref{wavelenths_1} contains more than one $\widetilde{J}$-orbit if and only if
\begin{equation}\label{maisOrbitas}
 \frac{\vert \holoL \vert}{\vert\Jup\vert}>
 \frac{\vert\Sigma_{k}\vert}{\vert\Jup   \cap\Sigma_{k}\vert} .
\end{equation}
\end{corollary}
\begin{proof}
	We know that the cardinal number of the orbit of $k$  by $\holoL $ is the index 
\begin{equation}\label{1.1}
\vert \holoL k\vert = \frac{\vert \holoL \vert}{\vert\Sigma_{k}\vert}.
\end{equation}  
The elements of $\holoL\cdot k$ that are projected into $\Jup\cdot P(k)$ are those in $ S_{k}\cdot k$, and this set coincides with $\Jup\cdot  k$ by item \eqref{p64} of Proposition~\ref{caracterizacao}.
Since the isotropy subgroup of $k$ in $\Jup  $ is $\Jup   \cap\Sigma_{k}$, then the number of different elements in 
$\holoL\cdot k$ whose projection lies in 
$\widetilde{J}\cdot P(k)$ is $\vert \Jup\vert/\vert \Jup   \cap\Sigma_{k}\vert$.
Therefore, there are points in $\holoL\cdot k$ whose projection does not lie in $\widetilde{J}\cdot P(k)$ if and only if
$$
 \frac{\vert \holoL \vert}{\vert\Sigma_{k}\vert}>
 \frac{\vert\Jup\vert}{\vert\Jup   \cap\Sigma_{k}\vert}
 $$
and this condition is equivalent to \eqref{maisOrbitas}.
\end{proof}

%The next corollary gives the number $r+1$ in  \eqref{wavelenths_1},  of different $\widetilde{J} $-orbits in $P(\holoL (\tilde{k},z))$, in terems of the order of the groups we use in this section.
%%	The number of different $\widetilde{J} $-orbits in $P(\holoL (\tilde{k},z))$ is given in the next lemma. The lemma gives a relation between the number $r$ in  \eqref{wavelenths_1} and the order of the groups we use in this section.
%
%\begin{corollary}\label{projected_ks}
%	Let $k = (\tilde{k},z) \in \LL^{*}$. Then, the number $r$ in \eqref{wavelenths_1} is 
%\[r \  =  \frac{\vert \holoL  \vert\cdot\vert\Jup   \cap\Sigma_{k}\vert}{\vert\Jup  \vert \cdot\vert \Sigma_{k}\vert} - 1.\]
%%where $\Sigma_{k} = \lbrace \delta \in J; \ \delta k = k\rbrace$ is the isotropy subgroup of $k$.
%\end{corollary}
%\begin{proof}
%	We know that the cardinal number of the orbit of $k$  by $\holoL $ is the index 
%\begin{equation}\label{1.1}
%\vert \holoL k\vert = \frac{\vert \holoL \vert}{\vert\Sigma_{k}\vert}.
%\end{equation}  
%	
%	By Theorem \ref{sepa}, 
%\begin{equation}\label{1.2}
%\vert \holoL k\vert = (r+1)\vert S_{k}k\vert.
%\end{equation}
%
%	Then combining \eqref{1.1} and \eqref{1.2}, we have
%\[r+1  = \frac{\vert \holoL \vert}{\vert S_{k}k\vert\cdot\vert\Sigma_{k}\vert}  =  \frac{\vert \holoL  \vert\cdot\vert\Jup   \cap\Sigma_{k}\vert}{\vert\Jup  \vert\cdot \vert \Sigma_{k}\vert} \]	
%where we are using the fact that $ S_{k}k = \Jup  k$ and that the isotropy subgroup of $k$ in $\Jup  $ is $\Jup   \cap\Sigma_{k}$.	
%\end{proof}

\section{A $\gamaL$-Irreducible decomposition}\label{secIrreducible}

	We are ready to give a decomposition of the projection of $\Gamma_{\LL}$-irreducible subspaces of $\XX_{\LL}$. 
	As before we use the notation $\widetilde{J} =\Pi_{y_0}(\holoL)$.
	The group $\widetilde{\Gamma}_{\widetilde{\LL}}=\RR^{n}/\widetilde{\LL} \ \dotplus \widetilde{J}$ is contained in  $\Gamma_{\widetilde{\LL}}$, but does not necessarily coincide with it, since the inclusion $\widetilde{J}\subset H_{\tilde\LL}$ may be strict.
This is the case in Example~\ref{ex2} above, other examples appear in  \cite{Olicalab15}.
%	Examples where this is the case are given in \cite{Olicalab15}.
	
\begin{theorem}\label{irreducible_projection}
Let $V$ be the $\gamaL$-irreducible subspace of $\XX_{\LL}$ generated by the orbit of $k\in\LL^*$.
The projection $\Pi_{y_{0}}(V)$ is a $\widetilde{\Gamma}_{\widetilde{\LL}}$-invariant subspace of $\XX_{\widetilde{\LL}}$.
If, moreover, condition \eqref{maisOrbitas} holds
%$$
% \frac{\vert \holoL \vert}{\vert\Jup\vert}>
% \frac{\vert\Sigma_{k}\vert}{\vert\Jup   \cap\Sigma_{k}\vert}
%$$
then $\Pi_{y_{0}}(V)$ is  the sum of more than one $\widetilde{\Gamma}_{\widetilde{\LL}}$-irreducible subspace,
for almost all values of $y_0\in\RR$. 
\end{theorem}	

\begin{proof}
First note that $\Pi_{y_0}\left(\omega_{(k_1,k_2)}\right)=c_0\omega_{k_1}$ with $c_0=\int_0^{y_0} \mathrm{e}^{2\pi i k_2 y} dy$ and thus, $c_0=0$ if and only if $k_2y_0\in\ZZ\setminus \{0\}$.
Therefore, $\omega_{k_1}(x)\in  \Pi_{y_{0}}(V)$ if and only if $k_1\in P\left(\holoL\cdot k\setminus Z_{y_0}\right)$
where
$$
Z_{y_0}=\{(x,ny_0)\in\RR^{n+1}: \ n\in\ZZ\setminus\{0\} \} .
$$
If $\holoL\cdot k\subset Z_{y_0}$ then it follows that  $ \Pi_{y_{0}}(V) = \lbrace 0\rbrace$, a subspace that is trivially $\widetilde{\Gamma}_{\widetilde{\LL}}$-invariant.

Suppose then that $\hat{k}\in \holoL\cdot k\setminus Z_{y_0}\ne\emptyset$ and hence $\omega_{k_1}(x)\in  \Pi_{y_{0}}(V)$.
If $\alpha\in\widetilde{J}$ then either $\alpha_+$ or $\alpha_-\in\holoL$, let $\alpha_*$ be the appropriate one.
Then $\alpha_* \hat{k}\not\in Z_{y_0}$, implying that $\omega_{\alpha k_1}(x)\in  \Pi_{y_{0}}(V)$.
Since from \eqref{gamaOfOmega} we know that
$\alpha \omega_{k_1}(x)=\omega_{\alpha k_1}(x)$,  this shows that $ \Pi_{y_{0}}(V)$ is 
$\widetilde{\Gamma}_{\widetilde{\LL}}$-invariant.

When condition \eqref{maisOrbitas} holds, by Corollary~\ref{moreProjectedOrbits} we know that $P(\holoL\cdot k)$ contains at least two distinct $\widetilde{J}$-orbits, say $\widetilde{J}\cdot P(u_1,z_1)$ and $\widetilde{J}\cdot P(u_2,z_2)$.
The set $K =\displaystyle \frac{1}{z_{1}}\ZZ\cup \frac{1}{z_{2}}\ZZ$ is a discrete subset of $\RR$
and for all $y_{0} \in \RR$ such that $y_{0} \notin K$, the projection $\Pi_{y_{0}}(V)$ has at least two irreducible components.
\end{proof}

\begin{corollary}
	Suppose that $V$ is a $\gamaL$-irreducible subspace of $\XX_{\LL}$ generated by the orbit of $k\in\LL^*$.
If there exist $(u_1,z_1)$ and $(u_2,z_2)$ in $\holoL\cdot k$ with $z_{1} \neq \pm z_{2}$, then for almost all $y_0\in\RR$ the projection of a single mode is a mode interaction. 
%	$(\widetilde{k},z)$. 
%	If there exist $z_{1} \neq \pm z_{2} \in \lbrace \holoL\cdot k\mid_{2}\rbrace$,
%	$z_{1} \neq \pm z_{2} \in \lbrace \holoL(\widetilde{k},z)\mid_{2}\rbrace$, 
%	generically 
%	the projection of a single mode is a mode interaction. 
\end{corollary}

\begin{proof}
If $z_{1} \neq \pm z_{2}$, then by item \eqref{p62} of Proposition~\ref{caracterizacao}, the $\widetilde{J}$-orbits of $u_1$ and $u_2$ are distinct, as in Figure~\ref{figJorbitas}.
Since $K =\displaystyle \frac{1}{z_{1}}\ZZ\cup \frac{1}{z_{2}}\ZZ$ is a discrete subset of $\RR$, for all $y_0\not\in K$
the projection $\Pi_{y_0}(V)$ has at least two $\widetilde{\Gamma}_{\widetilde{\LL}}$-irreducible components.
Hence, we have  an interaction of at least two modes.
%
%	Suppose that there exist $z_{1} \neq \pm z_{2} \in \lbrace \holoL(\widetilde{k},z)\mid_{2}\rbrace$. Then there exist at least two $\widetilde{\Gamma}_{\widetilde{\LL}}$-irreducible subspaces terms on the projection of $V$.
%	 
%	 Without loss of generality suppose  $z_{1} \neq \pm z_{2} \neq 0$. Then, the set $K =\displaystyle \frac{1}{z_{1}}\ZZ\cup \frac{1}{z_{2}}\ZZ$ is a discrete subset of $\RR$. Therefore, for all $y_{0} \in \RR$ such that $y_{0} \notin K$, we have an interaction of at least two modes.
\end{proof}

We return to the examples of Section~\ref{secL*}.

\setcounter{example}{1}
\addtocounter{example}{-1}
\begin{example}
For the simple cubic lattice $\LL_1$, it is always true that $\widetilde{\LL}_1=P(\widetilde{\LL}\cap\{y=0\})$, for any $y_0$.
Hence, the projection of a $\Gamma_{\LL_1}$-irreducible subspace of $\XX_{\LL_1}$ is never the trivial subspace.
The number of $\widetilde{J}$-irreducible components of $\Pi_{y_0}(V)$, where $V$ is generated by $k\in\LL_1^*$, is shown in
Table~\ref{tablePL1}  for each $k\in\LL_1^*$.
For instance,  when $k=(a,b,0)$, $a\ne b \in\ZZ\setminus \{0\}$ the projection $\widetilde{V}=\Pi_{y_0}(V)$ is the sum of three $\widetilde{\Gamma}_{\widetilde{\LL_1}}$-irreducible subspaces, generated by the orbits of $(a,0)$, of $(0,b)$ and of $(a,b)$.
For $k=(a,0,0)$, $a\ne 0$, the projection $\Pi_{y_0}(V)$ has the two irreducible components generated by the orbits of $(a,0)$ and of $(0,0)$, one of the components consists of constant functions. 
The only  case when the projection is irreducible is when  $V$ is generated by the orbit of $k=(a,a,a)$, $a\ne 0$, where the only irreducible component is generated by the orbit of $(a,a)$.
\hfill {\Large$\bullet$}
 \end{example}

%\begin{example}
%For the simple cubic lattice $\LL_1$, it is always true that $\widetilde{\LL}_1=P(\widetilde{\LL}\cap\{y=0\})$, for any $y_0$.
%Hence, the projection of a $\Gamma_{\LL_1}$-irreducible subspace of $\XX_{\LL_1}$ is never the trivial subspace.
%The number of $\widetilde{J}$-irreducible components of $\Pi_{y_0}(V_k)$ for each $k\in\LL_1^*$ is shown in
%Table~\ref{tablePL1}.
%For instance,  when $k=(a,b,0)$, $a\ne b \in\ZZ\setminus \{0\}$ the projection $\widetilde{V}=\Pi_{y_0}(V)$ is the sum of three $\widetilde{\Gamma}_{\widetilde{\LL_1}}$-irreducible subspaces, generated by the orbits of $(a,0)$, of $(0,b)$ and of $(a,b)$.
%For $k=(a,0,0)$, $a\ne 0$, the projection $P(V_k)$ has the two irreducible components generated by the orbits of $(a,0)$ and of $(0,0)$, one of the components consists of constant functions. 
%The only  case when the projection is irreducible is when  $V$ is generated by the orbit of $k=(a,a,a)$, $a\ne 0$, where the only irreducible component is generated by the orbit of $(a,a)$.
%\hfill {\Large$\square$}
% \end{example}

\begin{table}
\caption{For the simple cubic lattice $\LL_1$ in Example~\ref{ex1}, we know that $ {\vert \holoL \vert}/{\vert\Jup\vert}=3$ and for each $k\in\LL_1^*$ this table gives the numbers ${\vert\Sigma_{k}\vert}$ and ${\vert\Jup   \cap\Sigma_{k}\vert}$ and
 the number $N$ of non-trivial  $\widetilde{J}$-irreducible components of the projection of the $\Gamma_{\LL_1}$-irreducible subspace  $V\in\XX_{\LL_1}$ generated by the orbit of $k$.
Here $a \neq b \neq c\neq a$ with $a,b,c\in\ZZ\setminus\{0\}$.}\label{tablePL1}
%\caption{For the simple cubic lattice $\LL_1$ in Example~\ref{ex1}, we know that $ {\vert \holoL \vert}/{\vert\Jup\vert}=3$
%% generators $P(k)$ of 
%and for each $k\in\LL_1^*$ this table gives the numbers ${\vert\Sigma_{k}\vert}$ and ${\vert\Jup   \cap\Sigma_{k}\vert}$ and
% the number $N$ of non-trivial  $\widetilde{J}$-irreducible components of the projection of the $\Gamma_{\LL_1}$-irreducible subspace  $V_k\in\XX_{\LL_1}$ generated by the orbit of $k$.
%Here $a \neq b \neq c\neq a$ with $abc\neq 0$.}\label{tablePL1}
\begin{center}\begin{tabular}{l|l|l|l|l|l|l|l}
$k$& $(a,0,0)$&$(a,b,b)$&$(a,a,b)$&$(a,a,0)$&$(a,b,0)$&$(a,b,c)$&$(a,a,a)$\\
\hline
${\vert\Sigma_{k}\vert}$&8&2&2&4&2&1&6\\
\hline
${\vert\Jup   \cap\Sigma_{k}\vert}$&4&1&2&4&2&1&2\\
\hline
$N$&2&2&2&2&3&3&1\\
%\hline
%generators&$(a,0)$, $(0,0)$&$(a,b)$, $(b,b)$&$(a,a)$, $(a,b)$&$(a,a)$, $(a,0)$&$(a,b)$, $(a,0)$, $(b,0)$&$(a,b)$, $(a,c)$, $(b,c)$&$(a,a)$
\end{tabular}\end{center}
\end{table}

\begin{table}
\caption{For the rotated cubic lattice $\LL_2$ in Example~\ref{ex2}, and for $y_0= m\sqrt{6}/2$, $m\in\ZZ\setminus\{0\}$,
we have that $ {\vert \holoL \vert}/{\vert\Jup\vert}=4$.
For each $k\in\LL_1^*$ this table gives the transformed vector in $\LL_2^*$, the numbers ${\vert\Sigma_{k}\vert}$, the set $\Jup   \cap\Sigma_{k}$, the number $|\Sigma_k|/|\Sigma_k\cap\Jup|$ and the restriction on $y_0$.
If $V$ is the irreducible subspace generated by the orbit of $k$, then the only case where $\Pi_{y_0}(V)$ has only one irreducible component is $k=\left(a,\frac{\sqrt{3}a}{3},\frac{\sqrt{6}a}{3}\right)$. 
A small change in $y_0$ destroys this situation, see Table~\ref{tablePL2}.
Here $a \neq b \neq c\neq a$ with $a,b, c$ and $n\in\ZZ\setminus\{0\}$.}\label{table2PL2}
%\caption{For the rotated cubic lattice $\LL_2$ in Example~\ref{ex2}, and $y_0= m\sqrt{6}/2$, $m\in\ZZ\setminus\{0\}$,
%we have that $ {\vert \holoL \vert}/{\vert\Jup\vert}=4$.
%% generators $P(k)$ of 
%For each $k\in\LL_1^*$ this table gives the transformed vector in $\LL_2^*$, the numbers ${\vert\Sigma_{k}\vert}$, the set $\Jup   \cap\Sigma_{k}$, the number $|\Sigma_k|/|\Sigma_k\cap\Jup|$ and the restriction on $y_0$.
%The only case where $\Pi_{y_0}(V_k)$ has only one irreducible component is $k=\left(a,\frac{\sqrt{3}a}{3},\frac{\sqrt{6}a}{3}\right)$, but a small change in $y_0$ destroys this situation, see Table~\ref{tablePL2}.
%Here $a \neq b \neq c\neq a$ with $abc\neq 0$, $n\in\ZZ\setminus\{0\}$.}\label{table2PL2}
\begin{tabular}{clcccc}
%{l|l|c|lc|lc}
$k\in\LL_1^*$	&	$\sqrt{2}A k \in\LL_2^*$
%$\sqrt{2} B k$	
&	$|\Sigma_k|$	&	$\Sigma_k\cap\Jup$	&	$|\Sigma_k|/|\Sigma_k\cap\Jup|$	&	restriction	\\
\hline\hline
$(a,0,0)$	&	$\left(a,\frac{\sqrt{3}a}{3},\frac{\sqrt{6}a}{3}\right)$	&	8	&	$\{ Id_3,-R_{(0,1,1)}\}$	&	4	&	-	\\ \hline
$(a,b,b)$	&	$\left(a-b,\sqrt{3}\left(\frac{a}{3}+b\right),\frac{\sqrt{6}a}{3}\right)$	&	2	&	$\one$	&	2	&	$y_0\ne\frac{\sqrt{6}}{2a}n$	\\ \hline
$(0,a,a)$	&	$\left(-a,\sqrt{3}a,0\right)$	&	4	&	$\{ Id_3,R_{(0,1,1)}\}$	&	2	&	-	\\ \hline
$(a,b,0)$	&	$\left(a-b,\frac{\sqrt{3}}{3}(a+b),\frac{\sqrt{6}}{3}(a+b)\right)$	&	2	&	$\one$	&	2	&	$y_0\ne\frac{\sqrt{6}}{2(a+b)}n$	\\ \hline
$(a,b,c)$	&	$\left(a-b,\frac{\sqrt{3}}{3}(a+b+2c),\frac{\sqrt{6}}{3}(a+b-c)\right)$	&	1	&	$\one$	&	1	&	$y_0\ne\frac{\sqrt{6}}{2(a+b-c)}n$	\\ \hline
$(a,a,a)$	&	$\left(0,\frac{4\sqrt{3}a}{3},\frac{\sqrt{6}a}{3}\right)$	&	6	&	$\{ Id_3,R_{(1,-1,0)}\}$	&	3	&	$y_0\ne\frac{\sqrt{6}}{2a}n$	\\ \hline
$(a,b,a+b)$	&	$\left(a-b,\frac{2\sqrt{3}}{3}(a+b),0\right)$	&	1	&	$\one$	&	1	&	-	\\ \hline
$(a,0,a)$	&	$\left(a,\frac{2\sqrt{3}a}{3},0\right)$	&	4	&	$\{ Id_3,R_{(1,0,1)}\}$	&	2	&	-	\\ \hline
$(a,a,2a)$	&	$\left(0,\frac{4\sqrt{3}a}{3},0\right)$	&	2	&	$\{ Id_3,-R_{(1,-1,0)}\}$	&	1	&	-			\end{tabular}
\end{table}

Example~\ref{ex2} is  more typical, because the projection depends on $y_0$.
\begin{example}
The holohedry $H_{\LL_2}$ of the rotated cubic lattice $\LL_2$ has order 48.
The subgroup $\Jup$ has order 12 if $y_0= n\sqrt{6}/2$ with $n\in\ZZ\setminus\{0\}$, otherwise it has order 6.
Tables~\ref{table2PL2} and \ref{tablePL2} give all the information necessary to apply  Theorem~\ref{irreducible_projection} to this lattice.
For each $k$ there is at most a discrete set of values of $y_0$ for which a single mode is projected into a single mode.
Generically the projection is a mode interaction.
\hfill {\Large$\bullet$}
\end{example}

\begin{table}
\caption{For the rotated cubic lattice $\LL_2$ in Example~\ref{ex2}, and $y_0\ne m\sqrt{6}/2$, $m\in\ZZ\setminus\{0\}$,
we have that $ {\vert \holoL \vert}/{\vert\Jup\vert}=8$.
For each $k\in\LL_1^*$ this table gives the transformed vector in $\LL_2^*$, the numbers ${\vert\Sigma_{k}\vert}$, the set $\Jup   \cap\Sigma_{k}$, the number $|\Sigma_k|/|\Sigma_k\cap\Jup|$ and the restriction on $y_0$.
If $V$ is the irreducible  subspace
%subspae 
generated by the orbit of $k$, then the
 projection $\Pi_{y_0}(V)$ has always more than one irreducible component as long as the restriction on $y_0$ holds.
Here $a \neq b \neq c\neq a$ with $a,b, c$ and $n\in\ZZ\setminus\{0\}$.}\label{tablePL2}
%\caption{For the rotated cubic lattice $\LL_2$ in Example~\ref{ex2}, and $y_0\ne m\sqrt{6}/2$, $m\in\ZZ\setminus\{0\}$,
%we have that $ {\vert \holoL \vert}/{\vert\Jup\vert}=8$.
%% generators $P(k)$ of 
%For each $k\in\LL_1^*$ this table gives the transformed vector in $\LL_2^*$, the numbers ${\vert\Sigma_{k}\vert}$, the set $\Jup   \cap\Sigma_{k}$, the number $|\Sigma_k|/|\Sigma_k\cap\Jup|$ and the restriction on $y_0$.
%The projection $\Pi_{y_0}(V_k)$ has always more than one irreducible component as long as the restriction on $y_0$ holds.
%Here $a \neq b \neq c\neq a$ with $abc\neq 0$, $n\in\ZZ\setminus\{0\}$.}\label{tablePL2}
\begin{tabular}{clcccc}
$k\in\LL_1^*$	&	$\sqrt{2}A k \in\LL_2^*$
%$\sqrt{2} B k$	
&	$|\Sigma_k|$	&	$\Sigma_k\cap\Jup$	&	$|\Sigma_k|/|\Sigma_k\cap\Jup|$	&	restriction	\\
\hline\hline
$(a,0,0)$	&	$\left(a,\frac{\sqrt{3}a}{3},\frac{\sqrt{6}a}{3}\right)$	&	8	&	$\{ Id_3,-R_{(0,1,1)}\}$	&	4	&	$y_0\ne\frac{\sqrt{6}}{2a}n$	\\ \hline
$(a,b,b)$	&	$\left(a-b,\sqrt{3}\left(\frac{a}{3}+b\right),\frac{\sqrt{6}a}{3}\right)$	&	2	&	$\one$	&	2	&	$y_0\ne\frac{\sqrt{6}}{2a}n$	\\ \hline
$(0,a,a)$	&	$\left(-a,\sqrt{3}a,0\right)$	&	4	&	$\one$	&	4	&	-	\\ \hline
$(a,b,0)$	&	$\left(a-b,\frac{\sqrt{3}}{3}(a+b),\frac{\sqrt{6}}{3}(a+b)\right)$	&	2	&	$\one$	&	2	&	$y_0\ne\frac{\sqrt{6}}{2(a+b)}n$	\\ \hline
$(a,b,c)$	&	$\left(a-b,\frac{\sqrt{3}}{3}(a+b+2c),\frac{\sqrt{6}}{3}(a+b-c)\right)$	&	1	&	$\one$	&	1	&	$y_0\ne\frac{\sqrt{6}}{2(a+b-c)}n$	\\ \hline
$(a,a,a)$	&	$\left(0,\frac{4\sqrt{3}a}{3},\frac{\sqrt{6}a}{3}\right)$	&	6	&	$\one$	&	6	&	$y_0\ne\frac{\sqrt{6}}{2a}n$	\\ \hline
$(a,b,a+b)$	&	$\left(a-b,\frac{2\sqrt{3}}{3}(a+b),0\right)$	&	1	&	$\one$	&	1	&	-	\\ \hline
$(a,0,a)$	&	$\left(a,\frac{2\sqrt{3}a}{3},0\right)$	&	4	&	$\one$	&	4	&	-	\\ \hline
$(a,a,2a)$	&	$\left(0,\frac{4\sqrt{3}a}{3},0\right)$	&	2	&	$\{ Id_3,-R_{(1,-1,0)}\}$	&	1	&	-	\end{tabular}
\end{table}

\begin{figure}[htt]
\includegraphics[scale=0.25]{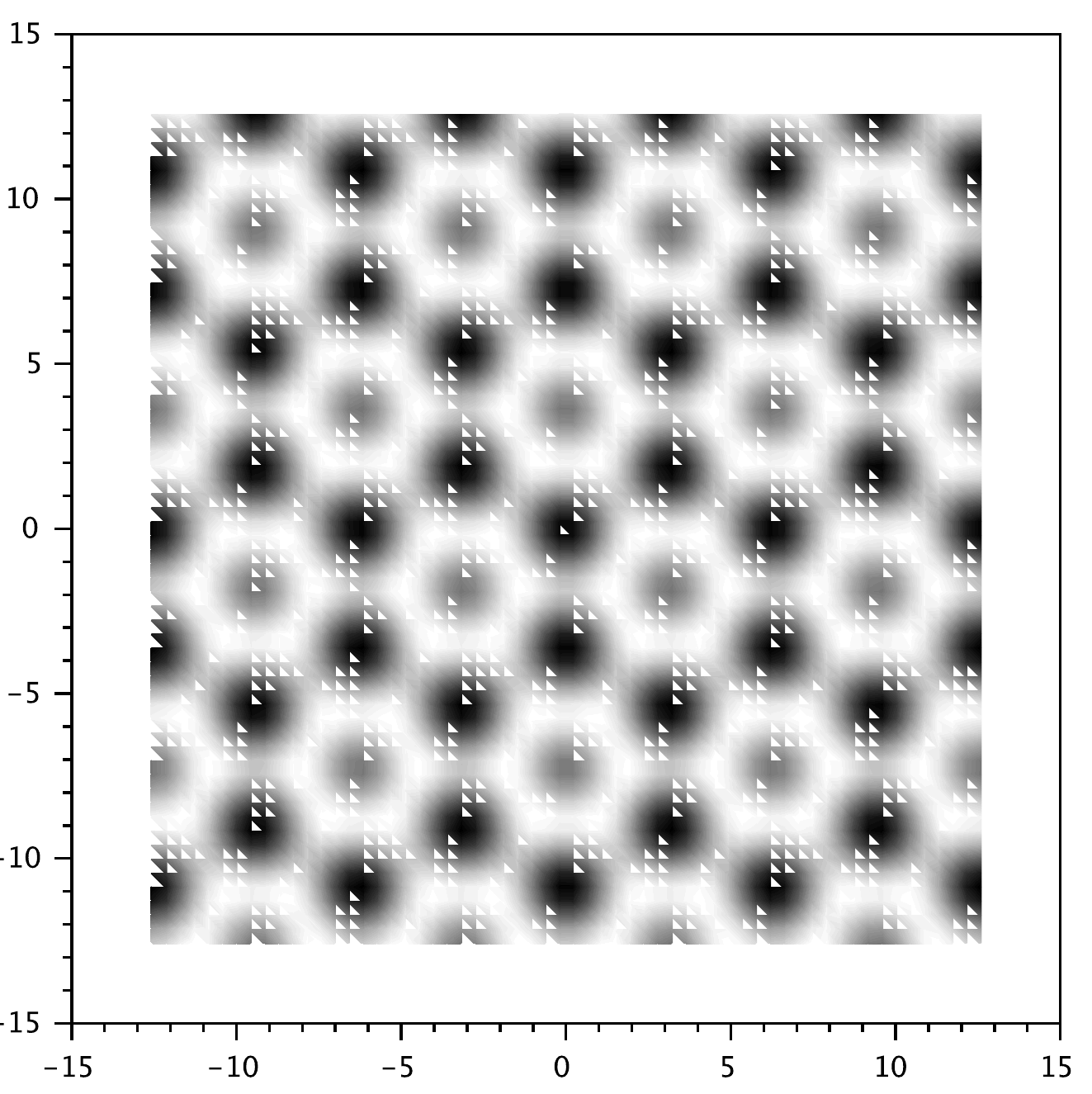}
\caption{Black eye pattern, obtained projecting the function $I_{(1/2,\sqrt{3}/2,0)}\in\XX_{\LL_3}$ with a band of projection of width $y_0=\sqrt{6}/4$. }
\label{figBlackEye}
\end{figure}

\begin{example}\label{exBCC}
The body centered cubic lattice $\LL_3$  was  used in \cite{gomes99} to obtain black-eye patterns by projection, 
like those of Figure~\ref{figBlackEye}.
It is generated by $(1/2,\sqrt{3}/2,0)$, $(1,0,0)$ and $(1/2,\sqrt{3}/6,-\sqrt{6}/12)$.
It contains the lattice $\LL_2$ of Example~\ref{ex2} and the two lattices share the same holohedry.
Generators for $\LL_3^*\subset\LL_2^*$ are $(1,\sqrt{3},0)$, $(2,0,0)$ and $(1,-\sqrt{3}/3,2\sqrt{6}/3)$.
Since $\LL_3^*$ does not contain any vector of the form $(a,\sqrt{3}a/3,\sqrt{6}a/3)$, then the projection of every $H_{\LL_3}$-orbit in $\LL_3^*$ contains more than one $\widetilde{J}$-orbit.
For instance (see Figure~\ref{figL3Jorbits}), the orbit $H_{\LL_3}\cdot (1,\sqrt{3},0)$ has 12 elements:
$$
%H_{\LL_3}\cdot (1,\sqrt{3},0)=
\left\{
\pm (1,\sqrt{3},0), \pm(2,0,0), \pm(-1,\sqrt{3},0),
\pm (1,\frac{\sqrt{3}}{3},\frac{-2\sqrt{6}}{3}), \pm (0,\frac{2\sqrt{3}}{3},\frac{2\sqrt{6}}{3}), \pm (-1,\frac{\sqrt{3}}{3},\frac{-2\sqrt{6}}{3})
\right\}
$$
and  for $y_0=\sqrt{6}n/4$, $n\in\ZZ\setminus\{0\}$ its projection consists of two $\widetilde{J}$-orbits, whereas for all other values of $y_0$ it is formed by three $\widetilde{J}$-orbits.

An irreducible subspace $V\subset\XX_{\LL_3}$ generated by the $H_{\LL_3}$-orbit of $k\in\LL_3$ contains a maximally symmetric function 
$$
I_k(x,y,z)=\sum_{\delta\in H_{\LL_3}} \omega_{\delta k}(x,y,z) .
$$
The subspace of $V$ generated by $I_k$ is the fixed-point subspace of the maximal subgroup $H_{\LL_3}\subset \Gamma_{\LL_3}$.
The Equivariant Branching Lemma  \cite[Lemma 1.31]{goluste02} may be applied to obtain patterns bifurcating into this subspace.
The pattern shown in Figure~\ref{figBlackEye} is the projection
$\Pi_{y_0}(I_{(1/2,\sqrt{3}/2,0)})$ for $y_0=\sqrt{6}/4$. 

For $\tilde{k}\in \widetilde{\LL}_3$, let $\widetilde{I}_{\tilde{k}}$ be given by
$$
 \widetilde{I}_{\tilde{k}}=\sum_{\beta\in\widetilde{J}} \omega_{\beta \tilde{k}}(x,y) .
$$

For most values  of $y_0$, the projected group $\widetilde{J}$ is isomorphic to the dihedral group $D_3$ of order 6.
The $H_{\LL_3}$-orbit of $(1/2,\sqrt{3}/2,0)$ is projected into three $\widetilde{J}$-orbits, shown in Figure~\ref{figL3Jorbits}.
This means that $\Pi_{y_0}(I_{(1/2,\sqrt{3}/2,0)})$ is a linear combination of the three $\widetilde{J}$-invariant functions
$\widetilde{I}_{(2,0)}$, $\widetilde{I}_{(1,\sqrt{3}/3)}$ and $\widetilde{I}_{(-1,-\sqrt{3}/3)}$, where the first function is real and  the third function is the complex conjugate of the second.  The patterns created by these functions are shown in the upper part of Figure~\ref{JL3star6figsCor}.
%Figure~\ref{JL3star6fgsBW} (Figure~\ref{JL3star6figsCor}).

For $y_0=\sqrt{6}n/12$, $n\in\ZZ\setminus\{0\}$, the projected group $\widetilde{J}$ is isomorphic to the dihedral group $D_6$ of order 12.
The $H_{\LL_3}$-orbit of $(1/2,\sqrt{3}/2,0)$ is projected into two $\widetilde{J}$-orbits, since the $\widetilde{J}$-orbits of $\pm(1,\sqrt{3}/3)$ coincide, as in  Figure~\ref{figL3Jorbits}.
Therefore  $\widetilde{I}_{(1,\sqrt{3}/3)}=\widetilde{I}_{(-1,-\sqrt{3}/3)}$,
This means that $\Pi_{\sqrt{6}/4}(I_{(1/2,\sqrt{3}/2,0)})$ is a linear combination of the two $\widetilde{J}$-invariant functions
$\widetilde{I}_{(2,0)}$ and $\widetilde{I}_{(1,\sqrt{3}/3)}$.
For the geometrical consequences, see Figure~\ref{JL3star6figsCor}.
%Figure~\ref{JL3star6fgsBW} (Figure~\ref{JL3star6figsCor}).
\hfill {\Large$\bullet$}
\end{example}

\begin{figure}
\includegraphics[scale=0.25]{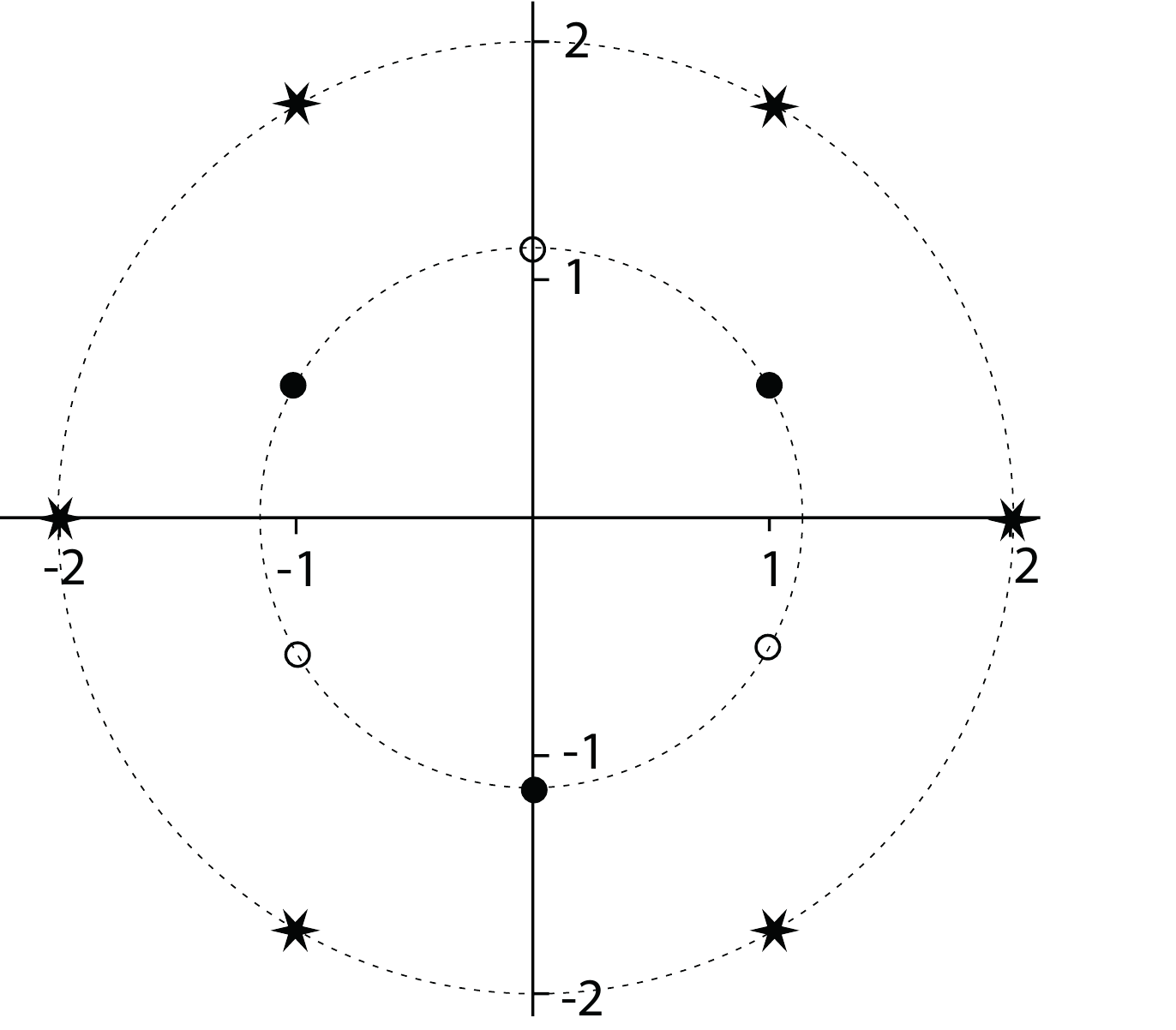}
\caption{The $\widetilde{J}$-orbits in the projection  $P(H_{\LL_3} (1,\sqrt{3},0))$ in Example~\ref{exBCC}.
For $y_0=\sqrt{6}n/12$, $n\in\ZZ\setminus \{0\}$ the projection decomposes into two $\widetilde{J}$-orbits, shown as marked points in the two dotted circles.
For other values of $y_0$ the  points with small norm  in the projection lie in two different $\widetilde{J}$-orbits, indicated here by black or white filled points. This gives rise to a 3-mode interaction.}
\label{figL3Jorbits}
\end{figure}

%\begin{figure}
%\includegraphics[scale=0.5]{JL3star6fgsBW.pdf}
%\caption{Top, left to right: level sets for $\widetilde{I}_{(2,0)}\in\RR$, 
%$\mathrm{Re}(\widetilde{I}_{(1,\sqrt{3}/3)})=\mathrm{Re}(\widetilde{I}_{(1,-\sqrt{3}/3)})$ and 
%$\mathrm{Im}(\widetilde{I}_{(1,\sqrt{3}/3)})=-\mathrm{Im}(\widetilde{I}_{(1,-\sqrt{3}/3)})$ 
%where $\widetilde{J}$ corresponds to $y_0\ne\sqrt{6}n/3$, $n\in\ZZ\setminus\{0\}$.
%Bottom: projections $\Pi_{y_0}\left( I_{(1/2,\sqrt{3}/2,0) }\right)$ for different $y_0$, from left to right:
%$y_0=2\sqrt{6}/3$, $y_0=\sqrt{6}/3$ and $y_0=8\sqrt{6}/15$.
%The pattern in  the centre is a perfect black eye, those on the sides are similar, but with less symmetry.
%Note that the patterns in the bottom are linear combinations of those on top, with different coefficients.
%}
%\label{JL3star6fgsBW}
%\end{figure}

\begin{figure}
\includegraphics[scale=0.5]{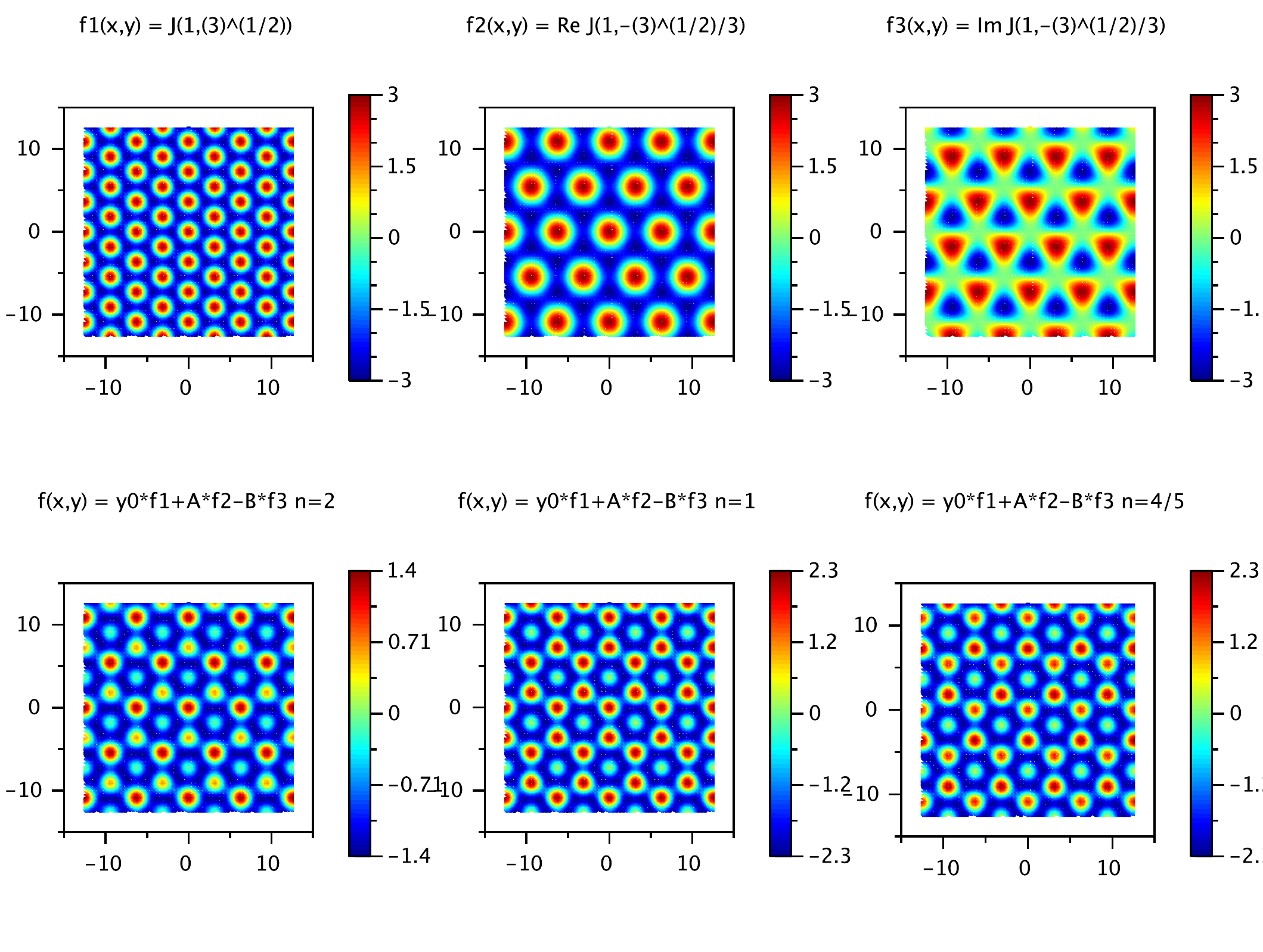}
\caption{Top, left to right: level sets for $\widetilde{I}_{(2,0)}\in\RR$, 
$\mathrm{Re}(\widetilde{I}_{(1,\sqrt{3}/3)})=\mathrm{Re}(\widetilde{I}_{(1,-\sqrt{3}/3)})$ and 
$\mathrm{Im}(\widetilde{I}_{(1,\sqrt{3}/3)})=-\mathrm{Im}(\widetilde{I}_{(1,-\sqrt{3}/3)})$ 
where $\widetilde{J}$ corresponds to $y_0\ne\sqrt{6}n/3$, $n\in\ZZ\setminus\{0\}$.
Bottom: projections $\Pi_{y_0}\left( I_{(1/2,\sqrt{3}/2,0) }\right)$ for different $y_0$, from left to right:
$y_0=2\sqrt{6}/3$, $y_0=\sqrt{6}/3$ and $y_0=8\sqrt{6}/15$.
The pattern in  the centre is the same as in Figure~\ref{figBlackEye}, a perfect black eye.
 Those on the sides are similar, but with less symmetry.
Note that the patterns at the bottom are linear combinations of those on top, with different coefficients.
}
\label{JL3star6figsCor}
\end{figure}

\section*{Acknowledgments}
	The authors would like to thank Professors Miriam Manoel and Gabriela Gomes for useful comments and suggestions.
	
	CMUP is supported by the European Regional Development Fund through the programme COMPETE and by the Portuguese Government through the Funda\c{c}\~ao para a Ci\^encia e a Tecnologia (FCT) under the project PEst-C/MAT/UI0144/2011. J.F. Oliveira was supported by a grant from the Conselho Nacional de Desenvolvimento Cient\'{i}fico e Tecnol\'{o}gico (CNPq) of Brazil.

\end{document}